\documentclass[11pt]{article}%
\usepackage{amsmath,amsfonts,amsthm,amssymb, graphicx}
\usepackage{subfig}
\usepackage{amsmath}
\usepackage{amsfonts}
\usepackage{amssymb}
\usepackage{graphicx}%
\providecommand{\U}[1]{\protect\rule{.1in}{.1in}}
\pdfoutput=1
\theoremstyle{plain}
\newtheorem{thm}{Theorem}

\newtheorem{lemma}{Lemma}
\newtheorem{corollary}{Corollary}

\begin{document}

\title{Two-parameter Sample Path Large Deviations for Infinite Server Queues}
\author{Blanchet, J., Chen, X., and Lam, H.}
\maketitle

\begin{abstract}
Let $Q_{\lambda}\left(  t,y\right)  $ be the number of people present at time
$t$ with $y$ units of remaining service time in an infinite server system with
arrival rate equal to $\lambda>0$. In the presence of a non-lattice renewal
arrival process and assuming that the service times have a continuous
distribution, we obtain a large deviations principle for $Q_{\lambda}\left(
\cdot\right)  /\lambda$ under the topology of uniform convergence on
$[0,T]\times\lbrack0,\infty)$. We illustrate our results by obtaining the most
likely path, represented as a surface, to ruin in life insurance portfolios,
and also we obtain the most likely surfaces to overflow in the setting of loss queues.

\end{abstract}

\section{Introduction}

The asymptotic analysis of queueing systems with many servers in heavy-traffic
has received substantial attention, especially in recent years. Among the
earliest references that come to mind in connection to this topic is the work
of \cite{Iglehart65} on heavy-traffic limits for the infinite-server queue.
Another highly influential paper in the area is \cite{pHAL81a} in the context
of many server Markovian queues, which introduced a scaling that is now known
as the \textquotedblleft Quality and Efficiency Driven\textquotedblright%
\ regime. The ideas in these papers have fueled more recent results in the
asymptotic analysis of many server systems such as: \cite{PuhRei00},
\cite{JelManMom04}, \cite{Reed09}, \cite{KaspiRamanan10},
\cite{KaspiRamanan11}, in the setting of many server queues, and
\cite{GlynnWhitt91}, \cite{DecMoy08}, \cite{PangWhitt10},
\cite{ReedTalreja2012}, in the setting of the infinite server queue. The
asymptotic analysis of queueing systems with many servers has been motivated
by applications in service engineering, in particular in the context of call
centers and health-care operations. Another set of application areas that is
also very relevant, but that is infrequently mentioned in the analysis of many
server systems is that of insurance mathematics. It is clear, for instance,
that a portfolio of insurance policies can be directly modeled as an infinite
server system; casting insurance portfolios in this framework is particularly
appealing in the setting of life insurance as we shall illustrate in Section
\ref{SectionExamples}.

So far most of the asymptotic analysis of many server systems has concentrated
mainly on fluid and heavy-traffic approximations. Meanwhile, the literature on
large deviations analysis for many server queues is not as extensive as that
of fluid or diffusion approximations; despite the fact that it is clearly of
interest to understand the large deviations behavior of these types of
systems. For instance, consider the consequences of dropping calls in an
emergency call center or being unable to satisfy the demand for critically ill
patients in the context of health-care applications. In the insurance setting,
for instance, it is of interest to estimate ruin probabilities and, perhaps
even more importantly, understanding the most likely path (or set of paths) to
ruin. Risk theory typically concentrates on ruin probabilities for aggregated
models, such as the classical ruin model (see \cite{AsmAlb10}); the results in
this paper, as we shall illustrate, provide a systematic way for assessing
ruin probabilities for a class of bottom-up models.

Our main contribution in this paper is to provide the first sample-path large
deviations analysis of the state descriptor of the infinite server queueing
model in heavy-traffic (i.e. as the arrival rate increases to infinity without
introducing any scaling on the service times). The statement of our main
result, which is given in Theorem \ref{main thm unbdd}, features a convenient
representation of a good large deviations rate function, under a strong
topology, that later we use in some applied settings of interest. For
instance, we compute the most likely path to overflow in a loss system, and
also the most likely path to ruin for a life insurance portfolio that embeds
an infinite server queue with a particular service cost structure. It is
important to emphasize that our results take advantage of a convenient
representation of the system's description that facilitates the representation
of the rate function; detailed discussion on this system's representation is
given in Section \ref{SubsectionAssumptions}. Previous large deviations
analysis of the infinite server queue has concentrated on queue-length
characteristics only; see, for instance, \cite{Glynn95a} who develops large
deviations for marginal quantities in the case of renewal arrivals, and
\cite{Zajic1998} who develops sample path large deviations for the queue
length process of infinite server queues in tandem in the case of Poisson arrivals.

Our large deviations analysis complement results on fluid analysis and
diffusion approximations recently obtained for infinite server systems. For
example, \cite{PangWhitt10} have shown that the state descriptor of the
infinite server queue, suitably parameterized in terms of a two-parameter
stochastic process, converges after centering and re-scaling to a Gaussian
process; see also \cite{ReedTalreja2012} who interpret the state descriptor of
the infinite server queue as a measure valued process acting on the space of
tempered distributions. These recent results, in turn, extend prior work by
\cite{GlynnWhitt91} in the context of discrete and bounded service time
distributions, and \cite{DecMoy08} for the case of Poisson arrivals.

The analysis of the infinite server queue is important as it serves as a
building block for other models of interest. For instance, in the setting of
loss models one can clearly couple the loss systems with associated infinite
server systems, and in the setting of many server queues \cite{Reed09} shows
how one can precisely understand queues with multiple servers as a
perturbation of infinite server queues. Furthermore, the infinite server model
is a classical model in queueing theory that serves as direct model in
important applications. Of particular interest to us, as mentioned earlier,
are the applications to insurance mathematics.

The rest of the paper is organized as follows. In Section
\ref{SectionNotation} we introduce our problem setting and provide a statement
of our main result. This is a fundamental section and it is divided into three
parts. We first shall introduce our assumptions and define our notation. Then
we provide the precise mathematical statement of our result, and, finally, we
will provide a heuristic argument that allows us to gain some intuition behind
our main result. The next two sections then provide the proof of our main
result. We first show our result for bounded service times in Section
\ref{SectionBdd}. Then, in Section \ref{SectionUnBdd}, we apply a truncation
argument to extend our result to unbounded service times. Finally, in Section
\ref{SectionExamples} we apply our result to computing the most likely paths
to rare events in the setting of loss queueing systems and also in the setting
of ruin probabilities for large life insurance portfolios.

\section{Assumptions, Notation, and Main Result\label{SectionNotation}}

The purpose of this section is threefold. First we shall clearly state our
assumptions and introduce necessary notation for our development. Second, we
shall explain the main large deviations result and provide a heuristic
derivation of the rate function that we obtain. Finally, we shall provide a
road map for the strategy behind the proof which will be presented in
subsequent sections.

\subsection{Assumptions and Notation\label{SubsectionAssumptions}}

We shall describe an underlying system corresponding to an arrival rate
$\lambda$. We call the system with $\lambda=1$, i.e. one customer per unit
time, our \textquotedblleft base system\textquotedblright; eventually we shall
send $\lambda$ to infinity in our asymptotic analysis. We collect our
assumptions as follows.\newline

\noindent\textbf{Assumptions and notation concerning the arrival process.} For
the base system, we assume the interarrival times are non-lattice, i.i.d.,
positive random variables $(U_{n}:n\geq1)$ with finite exponential moments in
a neighborhood of the origin; in precise words, $\kappa(\theta):=\log
Ee^{\theta U_{n}}<\infty$ for some $\theta>0$. In our $\lambda$-scaled system,
the arrivals come $\lambda$ times faster (i.e. the $n$-th interarrival times
becomes $U_{n}/\lambda)$. The associated logarithmic moment generating of the
$\lambda$-scaled service times is then $\kappa_{\lambda}(\theta):=\log
Ee^{\theta U_{n}/\lambda}=\kappa(\theta/\lambda)<\infty$ for some $\theta>0$.

The time at which the $n$-th arrival occurs in the base system is $A_{n}%
=U_{1}+...+U_{n}$ for $n\geq1$. We simply define $A_{0}:=0$ and then let
$N\left(  t\right)  :=\max\{n\geq0:A_{n}\leq t\}$ be the number of arrivals
that have occurred up to time $t$ in the base system. It is important to keep
in mind that $N\left(  \cdot\right)  $ increases by one unit at discontinuity
points since we are assuming that the $U_{n}$'s are positive.

Eventually, we shall increase the arrival rate, so it is sensible to define
$N_{\lambda}\left(  t\right)  :=N\left(  \lambda t\right)  $.

Define the so-called infinitesimal logarithmic moment generating function for
the arrival process via $\psi_{N}(\theta)=-\kappa^{-1}(-\theta)$ (see
\cite{Glynn95a}). This definition is motivated by the fact that
\[
\lim_{t\rightarrow\infty}\lambda^{-1}\log E\exp\left(  \theta\lbrack
N_{\lambda}\left(  t+\delta\right)  -N_{\lambda}\left(  t\right)  ]\right)
=\psi_{N}\left(  \theta\right)  \delta
\]
for any $\delta>0$. Since the $U_{n}$'s are positive with probability one we
have that $\psi_{N}\left(  \cdot\right)  $ is continuous and strictly convex
on the positive line. We also assume that $\psi_{N}\left(  \cdot\right)  $ is
continuously differentiable throughout $\mathbb{R}$. This assumption is
satisfied for most arrival processes, certainly for interarrival times that
are strictly positive and such that $\sup\{\kappa\left(  \theta\right)
:\kappa\left(  \theta\right)  <\infty\}=\infty$.\bigskip

\noindent\textbf{Assumptions and notation concerning the service times.} We
assume that the $n$-th customer that arrives to the base system (i.e. at time
$A_{n}$) brings up a service requirement of size $V_{n}$. The sequence
$(V_{n}:n\geq1)$ is assumed to be i.i.d. We write $F(x)=P(V_{n}\leq x)$ to
denote the associated distribution function evaluated at $x$, and set $\bar
{F}(x):=1-F(x)$ to be the tail distribution. Moreover, we assume that
$F\left(  \cdot\right)  $ is continuous.\newline

\noindent\textbf{Two-parameter representation of system status.} Let $\bar
{Q}_{\lambda}\left(  t,y\right)  $ denote the number of customers who arrived
before or at time $t$ and leave after time $y$ in the $\lambda$-scaled system.
In other words,%
\[
\bar{Q}_{\lambda}\left(  t,y\right)  =\bar{Q}_{\lambda}\left(  0,y-t\right)
+\sum_{n=1}^{N_{\lambda}\left(  t\right)  }I\left(  V_{n}+A_{n}/\lambda
>y\right)  .
\]
We shall assume that the system is initially empty at the beginning. This is
done for simplicity. Since we have infinitely many servers, we can incorporate
the initial configuration by keeping track of its evolution independently of
what occurs subsequently. Given our assumption of an initial empty system we
then have that $\bar{Q}_{\lambda}\left(  0,u\right)  =0$ for all $u\geq0$.
Note that for $t\in\lbrack0,T]$ and $y\geq0$,%
\begin{equation}
\bar{Q}_{\lambda}\left(  t,y\right)  =\bar{Q}_{\lambda}\left(  t\wedge
y,y\right)  +N_{\lambda}\left(  t\right)  -N_{\lambda}\left(  t\wedge
y\right)  . \label{EqQbarExt}%
\end{equation}

It is worth comparing the current system representation with the more common
one involving the quantity $Q_{\lambda}\left(  t,u\right)  $ defined as the
number of customers in the system currently at time $t$ who have residual
service time larger than $u\geq0$; more precisely, $Q_{\lambda}\left(
t,u\right)  =\bar{Q}_{\lambda}\left(  t,u+t\right)  $. These two system
representations are equivalent in the sense that $\left(  Q_{\lambda}\left(
t,u\right)  :t\in\lbrack0,T],u\geq0\right)  $ encodes the evolution of the
infinite server systems and thus, such evolution can be used in principle to
retrieve $(\bar{Q}_{\lambda}\left(  t,u\right)  :t\in\lbrack0,T],u\geq0)$. We
have chosen the representation based on $\bar{Q}_{\lambda}$ to facilitate the
representation of the rate function; a more detailed discussion is given
towards the end of Section \ref{SectionGuessing}. In addition, the
representation based on $\bar{Q}_{\lambda}$ allows to obtain a rich large
deviations principle to which one can apply the contraction principle directly
to several continuous functions of interest. For instance, it follows
immediately that the arrival process $N_{\lambda}\left(  t\right)  =\bar
{Q}_{\lambda}\left(  t,0\right)  $, and the departure process, $D_{\lambda
}\left(  t\right)  :=N_{\lambda}\left(  t\right)  -\bar{Q}_{\lambda}\left(
t,t\right)  $ are continuous functions under the topology that we consider
(and that we shall discuss in the next paragraphs). More applications of the
contraction principle will be discussed in Section \ref{SectionExamples}.
\bigskip

\noindent\textbf{Discussion about the topological space. }Let $\mathcal{D}%
=\{(t,y):0\leq t\leq T,y\geq0\}$ and let us write $||\cdot||_{\mathcal{C}}$ to
denote the supremum norm over any set $\mathcal{C}$. The space of functions
that we consider for our large deviations principle shall be denoted by
$L_{+,\infty}(\mathcal{D})$ and it corresponds to bounded functions with
domain in $\mathcal{D}$, such that $x\left(  0,u\right)  =0$ for $u\geq0$,
$x\left(  t,\cdot\right)  $ is non increasing, and $x\left(  t,\cdot\right)  $
vanishes at infinity. We will develop the large deviations principle for the
family of stochastic processes $(\bar{Q}_{\lambda}/\lambda:\lambda>0)$ on the
space $L_{+,\infty}(\mathcal{D})$ endowed with the topology generated by the
supremum norm. Following \cite{DEMZEI98} p. 4, the probability measures in
path space in our development are assumed to have been completed.

Our large deviations principle for $\bar{Q}_{\lambda}/\lambda$ immediately
implies in particular a large deviations principle in the Skorokhod topology
in the space $D_{D_{\mathbb{R}}[0,\infty)}[0,T]$ which is the space of
right-continuous-with-left-limits (RCLL) functions $x$, with domain on
$[0,T]$, that take values on the space of RCLL functions taking values on
$\mathbb{R}$. That is, on each time point $t$ in $x=\left(  x\left(  t\right)
:t\in\lbrack0,T]\right)  \in D_{D_{\mathbb{R}}[0,\infty)}[0,T]$ is a function
$x\left(  t\right)  \in D_{\mathbb{R}}[0,\infty)$. This is precisely the
topology considered in \cite{PangWhitt10}, who also provide a discussion on
the benefits of using this topology relative to other natural (but weaker)
alternative options (see Section 2.3 in \cite{PangWhitt10}).

An alternative approach that one might consider given the available results on
functional weak convergence analysis of the infinite server queue, such as
\cite{ReedTalreja2012}, is to interpret the space descriptor of the infinite
server queue as acting on the space of tempered distributions. We believe,
however, that this approach, although elegant, has important limitations in
terms of assumptions and the class of functions to which the contraction
principle can be directly applied to obtain other large deviations principles
of interest.




\subsection{Statement of Our Main result.}

We are now ready to state our main result. Let $\bar{q}:=\left(  \bar
{q}\left(  t,y\right)  :\left(  t,y\right)  \in\mathcal{D}\right)  \in
L_{+,\infty}(\mathcal{D})$. We say that $\bar{q}\in AC_{+}\left(
\mathcal{D}\right)  $ if the following conditions hold:

i) $\bar{q}$ is absolutely continuous on $\mathcal{D}$, and%
\[
\int_{0}^{\infty}\int_{0}^{\infty}\left\vert \frac{\partial^{2}}{\partial
t\ \partial y}\bar{q}(t,y)\right\vert dydt<\infty,
\]

ii) $\partial^{2}\bar{q}(t,y)/\partial t\ \partial y=0$ almost everywhere for
$\left(  t,y\right)  \in\{(t,y):0\leq y\leq t\leq T\}, $

iii) $\bar{q}\left(  0,y\right)  =0$ for $y\geq0$.

If $\bar{q}\in AC_{+}\left(  \mathcal{D}\right)  $, then we let $I\left(
\bar{q}\right)  $ be defined via%

\begin{equation}
\sup_{\theta(\cdot,\cdot)\in C_{b}[0,T]\times\lbrack0,\infty)}\int_{0}%
^{T}\left[  \int_{t}^{\infty}\theta(t,y-t)\left(  -\frac{\partial^{2}%
}{\partial t\ \partial y}\bar{q}(t,y)\right)  dy-\psi_{N}\left(  \log\left(
\int_{0}^{\infty}e^{\theta(t,y)}dF(y)\right)  \right)  \right]  dt \label{R}%
\end{equation}
where $C_{b}[0,T]\times\lbrack0,\infty)$ is the set of all bounded continuous
functions on $[0,T]\times\lbrack0,\infty)$. On the other hand, if $\bar{q}\in
L_{+,\infty}(\mathcal{D})$ fails to satisfy any of the conditions i)\ to iii),
simply let $I\left(  \bar{q}\right)  =\infty$.

We now can state our main result.

\begin{thm}
\label{main thm unbdd}Under the set of assumptions discussed in Section
\ref{SubsectionAssumptions}, $(\bar{Q}_{\lambda}/\lambda:\lambda>0)$ satisfies
a large deviations principle with good rate function $I\left(  \cdot\right)  $
on the space $(L_{+,\infty}(\mathcal{D})$, $\left\vert \left\vert
\cdot\right\vert \right\vert _{\mathcal{D}})$. In precise terms, for each open
set $O$ we have that%
\[
\underline{\lim}_{\lambda\rightarrow\infty}\log\frac{1}{\lambda}P(\bar
{Q}_{\lambda}/\lambda\in O)\geq-\inf_{q\in O}I\left(  q\right)  ,
\]
and for each closed set $C$%
\[
\overline{\lim}_{\lambda\rightarrow\infty}\log\frac{1}{\lambda}P(\bar
{Q}_{\lambda}/\lambda\in C)\leq-\inf_{q\in C}I\left(  q\right)  .
\]

\end{thm}

\bigskip

As an immediate corollary we obtain, as mentioned earlier, a large deviations
principle for $(Q_{\lambda}/\lambda:\lambda>0)$ under the Skorokhod topology
in the space $D_{D_{\mathbb{R}}[0,\infty)}[0,T]$, discussed in the previous
section and introduced in \cite{PangWhitt10}.


We shall explain the strategy behind the proof of Theorem \ref{main thm unbdd}%
. First, we shall introduce an auxiliary continuous process, $\tilde
{Q}_{\lambda}/\lambda$, which shall be defined later in Section
\ref{SubSectAuxiliary}. We will show that $\tilde{Q}_{\lambda}/\lambda$ is
exponentially equivalent to $\bar{Q}_{\lambda}/\lambda$. Second, in addition
to the assumptions imposed in Section \ref{SubsectionAssumptions} we will
assume that there exists a deterministic constant $M\in\left(  0,\infty
\right)  $ such that $P\left(  V_{n}\in\lbrack0,M]\right)  =1$. In the third
and last part of the argument we will relax this truncation assumption.

In turn, the first part of the argument (i.e. assuming truncation) is divided
into several steps.

The first step consists in developing the large deviations principle for
$\tilde{Q}_{\lambda}/\lambda$ with rate $I\left(  \cdot\right)  $ under the
topology of pointwise convergence using the Dawson-Gartner projective limit
theorem. The second step involves showing that $\tilde{Q}_{\lambda}/\lambda$
is exponentially tight as $\lambda\rightarrow\infty$ under the uniform
topology on the compact set $[0,T]\times\lbrack0,M]$. The third and last step
involves lifting the large deviations principle to the uniform topology.

During the second part we introduce an approximation scheme that proceeds by
ignoring the customers that arrive to the system with a service time larger
than $K$. Using a coupling argument, the process that is obtained using this
scheme is shown to be a good approximation to the original system for the
purpose of computing large deviations probabilities.

However, before we do this let us provide a heuristic argument in order to
guess the form of the rate function. Later we will explain what are the
technical difficulties that need to be addressed.

\subsubsection{Guessing the Rate Function: A Heuristic
Approach\label{SectionGuessing}}

One can take advantage of the point process representation of the input
process (i.e. the arrivals and the service times represented as a marked point
process). Let us start with the case of Poisson arrivals. We shall briefly
explain how to adapt the development that follows to the more general case of
renewal arrivals.

Consider the scaled system with arrival rate $\lambda$ and suppose that
$F\left(  \cdot\right)  $ has a density $f\left(  \cdot\right)  $. The amount
of customers that arrive during the time interval $[t,t+dt]$ that bring a
service requirement of size $[r,r+dr]$ is denoted by the quantity
$\mathcal{M}_{\lambda}\left(  t+dt,r+dr\right)  $, which is governed by a
Poisson distribution with rate $\lambda f\left(  r\right)  dtdr$. It follows
then by elementary considerations involving the Poisson distribution that
$\mathcal{M}_{\lambda}\left(  t+dt,r+dr\right)  /\lambda$ satisfies a large
deviations principle in the real line. In particular, we formally obtain that%
\[
P(\mathcal{M}_{\lambda}\left(  t+dt,r+dr\right)  /\lambda\approx\mu\left(
t,r\right)  dtdr)=\exp\left(  -\lambda J\left(  \mu\left(  t,r\right)
\right)  dtdr\right)  ,
\]
with
\[
J\left(  \mu\left(  t,r\right)  \right)  =\sup_{\eta\left(  t,r\right)  }%
[\eta\left(  t,r\right)  \mu\left(  t,r\right)  -\psi_{N}\left(  \eta\left(
t,r\right)  \right)  f\left(  r\right)  ],
\]
and $\psi_{N}\left(  \eta\right)  =\exp\left(  \eta\right)  -1$. The supremum
above is obtained formally with $\eta_{\ast}\left(  t,r\right)  :=\log
(\mu\left(  t,r\right)  /f\left(  r\right)  )$.

So, by pasting independent regions of the form $[t,t+dt]\times\lbrack r,r+dr]$
together one expects that the Poisson random measure $\mathcal{M}_{\lambda
}\left(  \cdot\right)  /\lambda$ would satisfy a large deviations principle
under a suitable topology, so that%
\[
P\left(  \mathcal{M}_{\lambda}\left(  A\times B\right)  /\lambda\approx
\int_{A\times B}\mu\left(  t,r\right)  dtdr\text{, for a large class }A\times
B\right)  \approx\exp\left(  -\lambda\mathbf{J}\left(  \mu\right)  \right)
\]
with%
\begin{align}
\mathbf{J}\left(  \mu\right)   &  =\int_{[0,T]\times\lbrack0,\infty)}%
[\eta_{\ast}\left(  t,r\right)  \mu\left(  t,r\right)  -\psi_{N}\left(
\eta_{\ast}\left(  t,r\right)  \right)  f\left(  r\right)  ]dtdr\nonumber\\
&  =\sup_{\eta\left(  \cdot,\cdot\right)  \in C_{b}[0,T]\times\lbrack
0,\infty)}\int_{[0,T]\times\lbrack0,\infty)}[\eta\left(  t,r\right)
\mu\left(  t,r\right)  -\psi_{N}\left(  \eta\left(  t,r\right)  \right)
f\left(  r\right)  ]dtdr. \label{LDP1}%
\end{align}

Now, observe that for all $y\geq t$%
\begin{equation}
\bar{Q}_{\lambda}\left(  t,y\right)  =\int_{[0,t]\times\lbrack y-s,\infty
)}\mathcal{M}_{\lambda}\left(  s+ds,r+dr\right)  , \label{QRep1}%
\end{equation}
and $\bar{Q}_{\lambda}\left(  0,y\right)  =0$ for $y\geq0$. Now, consider
\begin{equation}
\bar{q}(t,y)=\int_{0}^{t}\int_{y-s}^{\infty}v\left(  s,r\right)  drds.
\label{RepQ}%
\end{equation}
Note that
\[
\frac{\partial^{2}}{\partial y\partial t}\bar{q}(t,y)=-v\left(  t,y-t\right)
.
\]

Therefore, if $\bar{q}(\cdot,\cdot)$ is absolutely continuous and $\bar
{q}(0,y)=0$ for $y\geq0$, so that representation (\ref{RepQ}) is applicable,
one can formally compute the rate function of $\bar{Q}_{\lambda}\left(
\cdot,\cdot\right)  /\lambda$ evaluated at $\bar{q}(\cdot,\cdot)$ by
evaluating $\mathbf{J}\left(  \mu\right)  $ with
\[
\mu\left(  s,r\right)  =-\left.  \frac{\partial^{2}}{\partial y\partial t}%
\bar{q}(t,y)\right\vert _{t=s,y=s+r}
\]
for $s\in\lbrack0,T]$ and $r\in\lbrack0,\infty)$. In particular, this analysis
yields that $I\left(  \bar{q}\right)  $ should satisfy%
\begin{align*}
&  \sup_{\eta\left(  \cdot,\cdot\right)  \in C_{b}[0,T]\times\lbrack0,\infty
)}\int_{[0,T]\times\lbrack0,\infty)}[\eta\left(  s,r\right)  \mu\left(
s,r\right)  -\psi_{N}\left(  \eta\left(  s,r\right)  \right)  f\left(
r\right)  ]drds\\
&  =\sup_{\eta\left(  \cdot,\cdot\right)  \in C_{b}[0,T]\times\lbrack
0,\infty)}\int_{[0,T]\times\lbrack0,\infty)}\left[  \eta\left(  s,r\right)
\left(  -\frac{\partial^{2}}{\partial y\partial t}\bar{q}(s,s+r)\right)
-\psi_{N}\left(  \eta\left(  s,r\right)  \right)  f\left(  r\right)  \right]
drds\\
&  =\sup_{\eta\left(  \cdot,\cdot\right)  \in C_{b}[0,T]\times\lbrack
0,\infty)}\int_{[0,T]\times\lbrack s,\infty)}\left[  \eta\left(  s,u-s\right)
\left(  -\frac{\partial^{2}}{\partial y\partial t}\bar{q}(s,u)\right)
-(\exp(\eta\left(  s,u-s\right)  )-1)f\left(  u-s\right)  \right]  duds,
\end{align*}
which is, of course, equivalent to (\ref{R}) in the Poisson case assuming the
existence of a density $f\left(  \cdot\right)  $ for the distribution of the
service times. The previous form of the rate function was heuristically
obtained assuming that $y\geq t$. However, since all the information of the
infinite server queue is contained in the evolution of $(Q_{\lambda}\left(
t,y\right)  :t\in\lbrack0,T],u\geq0)$ and $Q_{\lambda}\left(  t,u\right)
=\bar{Q}_{\lambda}\left(  t,t+u\right)  $, we must have that the rate function
should be specified only over $\bar{q}\left(  t,y\right)  $, such that
$t\in\lbrack0,T]$ and $y\geq t$ so that $\partial^{2}\bar{q}(t,y)/\partial
y\partial t=0$ if $0\leq y\leq t\leq T$.

For the non-Poisson case one can argue using renewal arguments. We need to
compute the log-moment generating function of the vertical strip
$(\mathcal{M}_{\lambda}\left(  t+dt,r_{i}+dr\right)  :1\leq i\leq n)$, where
$r_{1}<r_{2}<...<r_{n}$ for an arbitrary partition $\left(  r_{i}:1\leq i\leq
n\right)  $. We obtain, using elementary properties of the multinomial
distribution together with an application of the key renewal theorem as in
\cite{Glynn95a},
\begin{align*}
&  E\left[  \exp\left(  \sum_{i=1}^{n}\theta\left(  t,r_{i}\right)
\mathcal{M}_{\lambda}\left(  t+dt,r_{i}+dr\right)  \right)  \right] \\
&  =E\left[  \left(  \sum_{i=1}^{n}\exp\left(  \theta\left(  t,r_{i}\right)
\right)  P\left(  V_{1}\in[r_{i},r_{i}+dr]\right)  \right)  ^{N\left(
\lambda(t+dt)\right)  -N\left(  \lambda t\right)  }\right] \\
&  =E\left[  \exp\left(  [N\left(  \lambda(t+dt)\right)  -N\left(  \lambda
t\right)  ]\log\left(  \sum_{i=1}^{n}\exp\left(  \theta\left(  t,r_{i}\right)
\right)  P\left(  V_{1}\in[r_{i},r_{i}+dr]\right)  \right)  \right)  \right]
\\
&  =\exp\left(  \lambda\psi_{N}\left(  \log\left(  \sum_{i=1}^{n}\exp\left(
\theta\left(  t,r_{i}\right)  \right)  P\left(  V_{1}\in\lbrack r_{i}%
,r_{i}+dr]\right)  \right)  \right)  +o\left(  \lambda\right)  \right)
\end{align*}
as $\lambda\rightarrow\infty$.

So, by pasting together vertical strips (i.e. ranging the parameter $t$) we
obtain that the family of random measures $\mathcal{M}_{\lambda}\left(
\cdot\right)  /\lambda$ is expected to satisfy a large deviations principle
under a suitable topology with rate function%
\[
\mathbf{J}\left(  \mu\right)  =\sup_{\theta\left(  \cdot,\cdot\right)  \in
C_{b}[0,T]\times\lbrack0,\infty)}\int_{[0,T]}\left[  \int_{0}^{\infty}%
\theta\left(  t,r\right)  \mu\left(  t,r\right)  dr-\psi_{N}\left(  \int
_{0}^{\infty}\exp\left(  \theta\left(  t,r\right)  \right)  dF\left(
r\right)  \right)  \right]  dt.
\]
The rest of the formal analysis proceeds similarly as in the Poisson case.

The formal argument just outlined, even if heuristic, suggests a potential
approach to developing sample path large deviations for $\bar{Q}_{\lambda
}/\lambda$. Namely, first develop a large deviations for the random measures
$\mathcal{M}_{\lambda}\left(  \cdot\right)  /\lambda$, and then apply the
contraction principle to obtain the desired large deviations result for
$\bar{Q}_{\lambda}/\lambda$. This approach, although intuitive, will not be
followed in our development. We found it easier to directly work with the
topology that we wish to impose. Part of the problem involved in making the
argument based on random measures rigorous in the setting of the topology that
is of interest to us is that indicator functions are not continuous, so the
contraction principle is not directly applicable if one is to endow the space
of measures with the weak convergence topology. Of course, one can proceed by
trying a different topology (stronger than weak convergence) or by trying to
use the extended contraction principle. However, the technical development, we
believe, would end up being more involved than the direct approach that we
will follow.

An additional concern that might arise at this point is our selection of
$\bar{Q}_{\lambda}/\lambda$ in order to represent the system status; as
opposed to $Q_{\lambda}/\lambda$, which might appear more natural at first
sight. Let us explain why $\bar{Q}_{\lambda}/\lambda$ is a more convenient
object to consider. Note that if $q\left(  s,r\right)  =\bar{q}(s,s+r)$, then
\begin{align*}
\frac{\partial^{2}}{\partial s\ \partial r}q\left(  s,r\right)   &
=\frac{\partial^{2}}{\partial s\ \partial r}\bar{q}(s,s+r)+\frac{\partial^{2}%
}{\partial r^{2}}\bar{q}(s,s+r)\\
&  =\frac{\partial^{2}}{\partial s\ \partial r}\bar{q}(s,s+r)+\frac
{\partial^{2}}{\partial r^{2}}q\left(  s,r\right)
\end{align*}
so%
\[
\frac{\partial^{2}}{\partial s\ \partial r}\bar{q}(s,s+r)=\frac{\partial^{2}%
}{\partial s\ \partial r}q\left(  s,r\right)  -\frac{\partial^{2}}{\partial
r^{2}}q\left(  s,r\right)  .
\]
Since $Q_{\lambda}\left(  t,u\right)  /\lambda=\bar{Q}_{\lambda}\left(
t,u+t\right)  /\lambda$ and our heuristic analysis suggests that the candidate
rate function of $\bar{Q}_{\lambda}\left(  t,y\right)  /\lambda$ is given by%
\[
\sup_{\theta(\cdot,\cdot)\in C_{b}[0,T]\times\lbrack0,\infty)}\int_{0}%
^{T}\left[  \int_{t}^{\infty}\theta(t,y-t)\left(  -\frac{\partial^{2}%
}{\partial t\ \partial y}\bar{q}(t,y)\right)  dy-\psi_{N}\left(  \log\left(
\int_{0}^{\infty}e^{\theta(t,y)}dF\left(  y\right)  \right)  \right)  \right]
dt,
\]
it is then sensible to conjecture, making $y=u+t$, a representation based on
$q\left(  t,u\right)  =\bar{q}(t,t+u)$ via%
\begin{align}
&  \sup_{\theta(\cdot,\cdot)\in C_{b}[0,T]\times\lbrack0,\infty)}\int_{0}%
^{T}\left[  \int_{0}^{\infty}\theta(t,u)\left(  -\frac{\partial^{2}}{\partial
t\ \partial y}\bar{q}(t,u+t)\right)  du-\psi_{N}\left(  \log\left(  \int
_{0}^{\infty}e^{\theta(t,u)}dF\left(  u\right)  \right)  \right)  \right]
dt\nonumber\\
&  =\sup_{\theta(\cdot,\cdot)\in C_{b}[0,T]\times\lbrack0,\infty)}\int_{0}%
^{T}\left[  \int_{0}^{\infty}\theta(t,u)\left(  -\frac{\partial^{2}}{\partial
t\ \partial u}q\left(  t,u\right)  +\frac{\partial^{2}}{\partial u^{2}%
}q\left(  t,u\right)  \right)  du-\psi_{N}\left(  \log\left(  \int_{0}%
^{\infty}e^{\theta(t,u)}dF\left(  u\right)  \right)  \right)  \right]  dt.
\label{Qrate}%
\end{align}
This representation, in turn, suggests that for the rate function to be finite
at $q\left(  \cdot\right)  $, one might need to impose as a necessary
condition the existence of $\partial^{2}q\left(  t,u\right)  /\partial^{2}u$.
Nevertheless, as we shall see in our examples, one might have a finite-valued
rate function even in cases in which $\partial q\left(  t,\cdot\right)
/\partial u$ is not even continuous for every value of $t\in\left(
0,T\right)  $.

\subsection{Construction of an Auxiliary Continuous Process
\label{SubSectAuxiliary}}

In order to prove Theorem \ref{main thm unbdd} we introduce an auxiliary
approximating continuous process, $\tilde{Q}_{\lambda}$, which shall be shown
to be exponentially equivalent to the process of interest $\bar{Q}_{\lambda}$
in the uniform norm. The construction of $\tilde{Q}_{\lambda}$ will be based
on polygonal interpolations, so it will be convenient to introduce some notation.

First, given $\left(  t,y\right)  $ and $\left(  t^{\prime},y^{\prime}\right)
$ where $t\neq t^{\prime}$ we write $Q_{\lambda}\left(  t,y\right)
\leftrightarrow Q_{\lambda}\left(  t^{\prime},y^{\prime}\right)  $ to denote
the straight line that joins the points $\left(  t,y,Q_{\lambda}\left(
t,y\right)  \right)  $ and $\left(  t^{\prime},y^{\prime},Q_{\lambda}\left(
t^{\prime},y^{\prime}\right)  \right)  $ in the associated three-dimensional space.

Now, given a sample path of the process $Q_{\lambda}\left(  \cdot\right)  $,
consider the set $\{t_{1},...,t_{m}\}$ of points corresponding to either
arrivals or departures in the interval $[0,T]$ (in increasing order); and put
$t_{0}=0$ and $t_{m+1}=T$. First let us consider $Q_{\lambda}(t,\cdot)$ for a
fixed time $t\in\{t_{0},...,t_{m+1}\}$. Let $\{y_{1}\left(  t\right)
,...,y_{n\left(  t\right)  }\left(  t\right)  \}$ be the set of
discontinuities of the function $Q_{\lambda}\left(  t,\cdot\right)  $ (recall
again that $Q_{\lambda}(t,\cdot)$ is a right continuous non-increasing step
function). Interpolate using straight lines forming the segments $Q_{\lambda
}(t,0)\leftrightarrow Q_{\lambda}(t,y_{1}\left(  t\right)  ),$ $Q_{\lambda
}(t,y_{1}\left(  t\right)  )\leftrightarrow Q_{\lambda}(t,y_{2}\left(
t\right)  ),\ldots,Q_{\lambda}(t,y_{n\left(  t\right)  -1}\left(  t\right)
)\leftrightarrow Q_{\lambda}(t,y_{n\left(  t\right)  }\left(  t\right)  )$.

The next step is to join the end points of these straight lines to the end
points of adjacent (suitably matched in the time axis) end points of straight
lines in order to form segments of adjacent planes. In order to do this
matching note that for each successive $t_{i}$ and $t_{i+1}$, either
$Q_{\lambda}(t_{i+1},\cdot)$ has one less discontinuous point than
$Q_{\lambda}\left(  t_{i},\cdot\right)  $ (i.e. a departure occurs at
$t_{i+1}$) or one more discontinuity point (i.e. an arrival occurs at
$t_{i+1}$); the exception is the last segment from $t_{m}$ to $t_{m+1}=T$,
where there might be no difference between the number of discontinuity points
between $Q_{\lambda}(t_{m},\cdot)$ and $Q_{\lambda}\left(  t_{m+1}%
,\cdot\right)  $. Note that batch arrivals are not possible since the
interarrival times are positive.

According to the notation introduced earlier for discontinuity points,
$y_{1}\left(  t_{i}\right)  ,...,y_{n\left(  t_{i}\right)  }\left(
t_{i}\right)  $ are the discontinuous points of $Q_{\lambda}(t_{i},\cdot)$
with corresponding values $Q_{\lambda}(t_{i},y_{1}\left(  t_{i}\right)  )$,
$Q_{\lambda}(t_{i},y_{2}\left(  t_{i}\right)  )$, $\ldots$, $Q_{\lambda}%
(t_{i},y_{n\left(  t_{i}\right)  }\left(  t_{i}\right)  )$. We will explain
how to joint discontinuity points of $Q_{\lambda}\left(  t_{i},\cdot\right)  $
with those from $Q_{\lambda}\left(  t_{i+1},\cdot\right)  $.

Suppose a departure occurs at time $t_{i+1}$. Then we can label the
discontinuous points of $Q_{\lambda}(t_{i+1},\cdot)$ as $y_{1}\left(
t_{i+1}\right)  ,...,y_{n\left(  t_{i+1}\right)  }\left(  t_{i+1}\right)  $,
with $n\left(  t_{i+1}\right)  =n\left(  t_{i}\right)  -1$. We form a set of
straight lines $Q_{\lambda}(t_{i},0)\leftrightarrow Q_{\lambda}(t_{i+1}%
,0)\leftrightarrow Q_{\lambda}(t_{i},y_{1}\left(  t_{i}\right)
)\leftrightarrow Q_{\lambda}(t_{i+1},y_{1}\left(  t_{i+1}\right)
)\leftrightarrow Q_{\lambda}(t_{i},y_{2}\left(  t_{i}\right)  )\leftrightarrow
Q_{\lambda}(t_{i+1},y_{2}\left(  t_{i+1}\right)  )\leftrightarrow
...\leftrightarrow$ $Q_{\lambda}(t_{i+1},y_{n\left(  t_{i+1}\right)  }\left(
t_{i+1}\right)  )\leftrightarrow Q_{\lambda}(t_{i},y_{n\left(  t_{i}\right)
}\left(  t_{i}\right)  )$ in a zig-zag manner; together with another set of
straight lines $Q_{\lambda}(t_{i},0)\leftrightarrow Q_{\lambda}(t_{i}%
,y_{1}\left(  t_{i}\right)  )\leftrightarrow\ldots\leftrightarrow Q_{\lambda
}(t_{i},y_{n\left(  t_{i}\right)  }\left(  t_{i}\right)  )$, and also the set
of straight lines $Q_{\lambda}(t_{i+1},0)\leftrightarrow Q_{\lambda}%
(t_{i+1},y_{1}\left(  t_{i+1}\right)  )\leftrightarrow\ldots\leftrightarrow
Q_{\lambda}(t_{i+1},y_{n\left(  t_{i+1}\right)  }\left(  t_{i+1}\right)  )$.
These three sets describe a series of adjacent triangular planar sections
which jointly form a continuous surface.

Similarly, suppose that an arrival occurs at time $t_{i+1}$. Then we can label
the discontinuous points of $Q_{\lambda}(t_{i+1},\cdot)$ as $Q_{\lambda
}(t_{i+1},y_{1}\left(  t_{i+1}\right)  )$, $\ldots$, $Q_{\lambda}%
(t_{i+1},y_{n\left(  t_{i+1}\right)  }\left(  t_{i+1}\right)  )$, with
$n\left(  t_{i+1}\right)  =n\left(  t_{i}\right)  +1$. We then form the set of
straight lines $Q_{\lambda}(t_{i+1},0)\leftrightarrow Q_{\lambda}%
(t_{i},0)\leftrightarrow Q_{\lambda}(t_{i+1},y_{1}\left(  t_{i+1}\right)
)\leftrightarrow Q_{\lambda}(t_{i},y_{1}\left(  t_{i}\right)  )\leftrightarrow
Q_{\lambda}(t_{i+1},y_{2}\left(  t_{i+1}\right)  )\leftrightarrow
...\leftrightarrow$ $Q_{\lambda}(t_{i},y_{n\left(  t_{i}\right)  }\left(
t_{i}\right)  )\leftrightarrow Q_{\lambda}(t_{i+1},y_{n\left(  t_{i+1}\right)
}\left(  t_{i+1}\right)  )$. Again, together with a second set of straight
lines $Q_{\lambda}(t_{i},0)\leftrightarrow Q_{\lambda}(t_{i},y_{1}\left(
t_{i}\right)  )\leftrightarrow\ldots\leftrightarrow Q_{\lambda}(t_{i}%
,y_{n\left(  t_{i}\right)  }\left(  t_{i}\right)  )$, and a third set of
straight lines, namely $Q_{\lambda}(t_{i+1},0)\leftrightarrow Q_{\lambda
}(t_{i+1},y_{1}\left(  t_{i+1}\right)  )\leftrightarrow\ldots\leftrightarrow
Q_{\lambda}(t_{i+1},y_{n(t_{i+1})}(t_{i+1}))$. These three sets of straight
lines, once again describe a series of adjacent triangular planar sections
which jointly form a continuous surface. The last time interval from $t_{m}$
to $T$ is dealt with similarly, with perhaps one less triangle formed if
$n\left(  t_{m}\right)  =n\left(  T\right)  $.

The continuous function $(Q_{\lambda}^{\ast}(t,y):0\leq t\leq T,y\geq0)$ is
defined by concatenating all these adjacent triangular planar regions as one
varies $t_{i}$ and $t_{i+1}$ for $i\in\{0,1,...,m\}$, and setting $Q_{\lambda
}^{*}(t,y)=0$ for the region where $y$ is beyond the boundary of the last
triangular plane i.e. beyond the lines $Q_{\lambda}(t_{i},y_{n(t_{i})}%
(t_{i}))\leftrightarrow Q_{\lambda}(t_{i+1},y_{n(t_{i+1})}(t_{i+1}%
)),i\in\{0,1,...,m\}$. It is immediate from the previous construction, and the
fact that $Q_{\lambda}(t,\cdot)$ is non increasing, that $Q_{\lambda}^{\ast
}(t,\cdot)$ is also non increasing for each $t\in\lbrack0,T]$.

Then, we define our auxiliary process $\tilde{Q}_{\lambda}\left(  t,y\right)
$ for $y\geq t$ via
\begin{equation}
\tilde{Q}_{\lambda}\left(  t,y\right)  =Q_{\lambda}^{\ast}(t,y-t).
\label{DefQtilde}%
\end{equation}
In order to define $\tilde{Q}_{\lambda}\left(  t,y\right)  $ for $0\leq y\leq
t\leq T$, first let $\tilde{N}_{\lambda}\left(  \cdot\right)  $ be the
continuous process obtained by the polygonal interpolation of $N_{\lambda
}\left(  \cdot\right)  $, so that $\tilde{N}_{\lambda}\left(  0\right)  =0$
and $\tilde{N}_{\lambda}\left(  A_{k}/\lambda\right)  =N_{\lambda}\left(
A_{k}/\lambda\right)  $ for all $k\geq1$. Then, define%
\[
\tilde{Q}_{\lambda}\left(  t,y\right)  =\tilde{Q}_{\lambda}\left(  t\wedge
y,y\right)  +\tilde{N}_{\lambda}\left(  t\right)  -\tilde{N}_{\lambda}\left(
t\wedge y\right)  ,
\]
analogous to (\ref{EqQbarExt}). Observe that%

\begin{equation}
||\tilde{Q}_{\lambda}-\bar{Q}_{\lambda}||_{\mathcal{D}}\leq4, \label{Bnd1}%
\end{equation}
where $||\cdot||_{\mathcal{D}}$ represents the uniform norm over the set
$\mathcal{D}$.

\section{Bounded Service Times\label{SectionBdd}}

In addition to the assumptions imposed in Section \ref{SectionNotation} here
we also assume that $P\left(  V_{n}\in\lbrack0,M]\right)  =1$ for $M\in\left(
0,\infty\right)  $.

We define $\mathcal{D}_{M}=\{(t,u):0\leq t\leq T,0\leq u\leq M+T\}$ and let
$C_{+}(\mathcal{D}_{M})$ be the space of functions $(x\left(  t,u\right)
:\left(  t,u\right)  \in\mathcal{D}_{M})$ such that $x\left(  \cdot\right)  $
is continuous in both components, $x\left(  0,u\right)  =0$ for $u\geq0$, and
$x\left(  t,\cdot\right)  $ vanishes at infinity. If in addition, $x\left(
\cdot,\cdot\right)  $ is absolutely continuous, and $\partial x\left(
t,y\right)  /\partial t\partial y=0$ almost everywhere on $0\leq y\leq t\leq
T$, we say that $x\left(  \cdot,\cdot\right)  \in$ $AC_{+}(\mathcal{D}_{M})$.

Our initial goal is to obtain a large deviations principle for $(\tilde
{Q}_{\lambda}/\lambda:\lambda>0)$ as $\lambda\rightarrow\infty$ on the space
$(C_{+}(\mathcal{D}_{M}),\left\vert \left\vert \cdot\right\vert \right\vert
_{\mathcal{D}_{M}})$; we then will use (\ref{Bnd1}) to obtain the
corresponding large deviations principle for $(\bar{Q}_{\lambda}%
/\lambda:\lambda>\infty)$.

We start by deriving a large deviations principle in the topology of pointwise
convergence. The proof of this result will be given at the end of this section.

\begin{lemma}
Let $\mathcal{X}$ consist of all the maps from $\mathcal{D}_{M}$ to
$\mathbb{R}$, and we equip $\mathcal{X}$ with the topology of pointwise
convergence on $\mathcal{D}_{M}$. Then $\tilde{Q}_{\lambda}/\lambda$ satisfies
a large deviations principle with good rate function $I(\bar{q})$ defined by
\begin{equation}
\sup_{\theta(\cdot,\cdot)\in\mathcal{C}[0,T]\times\lbrack0,M]}\int_{0}%
^{T}\left[  \int_{t}^{M+t}\theta(t,y-t)\left(  -\frac{\partial^{2}}{\partial
t\ \partial y}\bar{q}(t,y)\right)  dy-\psi_{N}\left(  \log\left(  \int_{0}%
^{M}e^{\theta(t,y)}dF\left(  y\right)  \right)  \right)  \right]  dt
\label{rate}%
\end{equation}
if $\bar{q}(\cdot)\in AC_{+}(\mathcal{D}_{M})$, and $I(\bar{q})=\infty$
otherwise. Here $\mathcal{C}[0,T]\times\lbrack0,M]$ denotes the set of all
continuous functions on $[0,T]\times\lbrack0,M]$.
\label{projection limit}
\end{lemma}

\bigskip

In order to lift the large deviations principle indicated in Lemma
\ref{projection limit} to the uniform topology we need the following result on
exponential tightness; we shall also give the proof of this result at the end
of this section.

\begin{lemma}
$\tilde{Q}_{\lambda}/\lambda$ is exponentially tight in $C_{+}(\mathcal{D}%
_{M})$ equipped with the topology of uniform convergence. \label{tightness}
\end{lemma}

\bigskip

Using the previous two lemmas we are ready to state and prove the main result
of this section, which is a version of Theorem \ref{main thm unbdd} for the
case of bounded service times.

\begin{thm}
\label{ThmBdd1} $\tilde{Q}_{\lambda}/\lambda$ satisfies a large deviations
principle with good rate function defined in \eqref{rate} under the uniform
topology on $\mathcal{D}_{M}$.
\end{thm}

\begin{proof}
Since the domain of $I\left(  \cdot\right)  $ is a subset of $C_{+}\left(
\mathcal{D}_{M}\right)  $, and $\tilde{Q}_{\lambda}/\lambda\in C_{+}\left(
\mathcal{D}_{M}\right)  $ with probability 1, the large deviations principle
in Lemma \ref{projection limit} holds in the space $C_{+}\left(
\mathcal{D}_{M}\right)  $ with pointwise topology, (Lemma 4.1.5 (b) in
\cite{DEMZEI98}). Since by Lemma \ref{tightness} $\tilde{Q}_{\lambda}/\lambda$
is exponentially tight in $\left(  C_{+}\left(  \mathcal{D}_{M}\right)
,\left\vert \left\vert \cdot\right\vert \right\vert _{\mathcal{D}_{M}}\right)
$ the same large deviations principle holds in $\left(  C_{+}\left(
\mathcal{D}_{M}\right)  ,\left\vert \left\vert \cdot\right\vert \right\vert
_{\mathcal{D}_{M}}\right)  $ (Corollary 4.2.6 in \cite{DEMZEI98}) and the
result follows.
\end{proof}

\bigskip

As a corollary of the previous theorem we obtain that $(\bar{Q}_{\lambda
}/\lambda:\lambda>0)$ satisfies a large deviations principle on $(L_{+,\infty
}(\mathcal{D}_{M})$, $\left\vert \left\vert \cdot\right\vert \right\vert
_{\mathcal{D}_{M}})$

\begin{corollary}
\label{CorMainBdd}The process $(\bar{Q}_{\lambda}/\lambda:\lambda>0)$
satisfies a large deviations principle on $(L_{+,\infty}(\mathcal{D}_{M})$,
$\left\vert \left\vert \cdot\right\vert \right\vert _{\mathcal{D}_{M}})$ with
rate function $I\left(  \cdot\right)  $ defined in \eqref{rate}.
\end{corollary}

\begin{proof}
First we verify that $Q_{\lambda}/\lambda$ and $\tilde{Q}_{\lambda}/\lambda$
are exponentially equivalent according to Definition 4.2.10 in \cite{DEMZEI98}%
). Since the laws of $(Q_{\lambda}/\lambda,\tilde{Q}_{\lambda}/\lambda)$ are
induced by a separable stochastic process and the underlying topology is
induced by the uniform norm, the set%
\[
\{\omega:||Q_{\lambda}/\lambda-\tilde{Q}_{\lambda}/\lambda||_{\mathcal{D}_{M}%
}>\eta\}
\]
is Borel measurable (see Remark b) following Definition 4.2.10 in
\cite{DEMZEI98}). Now recall that by the construction of $\tilde{Q}_{\lambda}$
that $\Vert Q_{\lambda}-\tilde{Q}_{\lambda}\Vert\leq4$ a.s. Hence for any
$\eta>0$,
\[
P(||Q_{\lambda}/\lambda-\tilde{Q}_{\lambda}/\lambda||_{\mathcal{D}_{M}}%
>\eta)=0
\]
for large enough $\lambda$. Hence
\[
\limsup_{\lambda\rightarrow\infty}\frac{1}{\lambda}\log P(||Q_{\lambda
}/\lambda-\tilde{Q}_{\lambda}/\lambda||_{\mathcal{D}_{M}}>\eta)=-\infty.
\]
The result then follows by applying Theorem 4.2.13 in \cite{DEMZEI98}.
\end{proof}

\subsection{Proofs of Technical Results\label{SectionAppendix1}}

Finally, we provide the proof of Lemmas \ref{projection limit} and
\ref{tightness}.

We start with Lemma \ref{projection limit} which takes advantage of the
Dawson-Gartner projective limit theorem and thus requires that we obtain an
auxiliary large deviations principle for finite dimensional objects defined
via
\begin{align}
\Delta_{ij}\left(  \lambda\right)   &  =\sum_{k=N_{\lambda}(t_{i-1}%
)+1}^{N_{\lambda}(t_{i})}I(y_{j-1}<A_{k}/\lambda+V_{k}\leq y_{j}%
)\label{DLTA}\\
&  =\bar{Q}_{\lambda}(t_{i},y_{j-1})-\bar{Q}_{\lambda}(t_{i},y_{j})-\bar
{Q}_{\lambda}(t_{i-1},y_{j-1})+\bar{Q}_{\lambda}(t_{i-1},y_{j}),\nonumber
\end{align}
for $t_{i-1}<t_{i}$, and $y_{j-1}<y_{j}$.

\begin{lemma}
\label{LemAuxLDPFinite}For $0=t_{0}<t_{1}<t_{2}<...<t_{m}\leq T$ and
$0=y_{0}<y_{1}<...<y_{n}<y_{n+1}=T+M$, $(\Delta_{ij}(\lambda)/\lambda:1\leq
i\leq m,1\leq j\leq n+1)$ possesses a large deviations principle with a good
rate function
\[
\sup_{\theta_{ij}:1\leq i\leq m,1\leq j\leq n+1}\sum_{i=1}^{m}\sum_{j=1}%
^{n+1}\theta_{ij}\delta_{ij}-\sum_{i=1}^{m}\int_{t_{i-1}}^{t_{i}}\psi
_{N}\left(  \log\sum_{j=1}^{n+1}e^{\theta_{ij}}P(y_{j-1}-u<V_{1}\leq
y_{j}-u)\right)  du.
\]

\end{lemma}

\begin{proof}
[Proof of Lemma \ref{LemAuxLDPFinite}]We use that $\psi_{N}(\cdot)$ is
continuously differentiable over $\mathbb{R}$. Since $U_{i}$ are non-lattice,
the key renewal theorem implies that for any set of $0\leq t_{0}<t_{1}%
<t_{2}<\cdots<t_{m}$,
\[
\lim_{\lambda\rightarrow\infty}\frac{1}{\lambda}\log E\exp\left\{  \sum
_{i=1}^{m}\theta_{i}(N_{\lambda}(t_{i})-N_{\lambda}(t_{i-1}))\right\}
=\sum_{i=1}^{m}\psi_{N}(\theta)(t_{i}-t_{i-1})
\]
for any $\theta\in\mathbb{R}$. Then, from \cite{Glynn95a}, the Gartner-Ellis
limit of $(\Delta_{ij}(\lambda):1\leq i\leq m,1\leq j\leq n+1)$ satisfies
\begin{align*}
\Lambda(\Theta)  &  =\lim_{\lambda\rightarrow\infty}\frac{1}{\lambda}\log
E\exp\left\{  \sum_{i=1}^{m}\sum_{j=1}^{n+1}\theta_{ij}\Delta_{ij}%
(\lambda)\right\} \\
&  =\sum_{i=1}^{m}\int_{t_{i-1}}^{t_{i}}\psi_{N}\left(  \log\sum_{j=1}%
^{n+1}e^{\theta_{ij}}P(y_{j-1}-u<V_{1}\leq y_{j}-u)\right)  du
\end{align*}
is finite for any $\Theta:=(\theta_{i,j}:1\leq i\leq m,1\leq j\leq n+1)$.
Moreover, for any $t_{i-1}<u\leq t_{i}$,
\begin{align*}
&  \left\vert \frac{\partial}{\partial\theta_{ij}}\psi_{N}\left(  \log
\sum_{k=1}^{n+1}e^{\theta_{ik}}P(y_{k-1}-u<V_{1}\leq y_{k}-u)\right)
\right\vert \\
=  &  \left\vert \psi_{N}^{\prime}\left(  \log\sum_{k=1}^{n+1}e^{\theta_{ik}%
}P(y_{k-1}-u<V_{1}\leq y_{k}-u)\right)  \right\vert \cdot\frac{e^{\theta_{ij}%
}P(y_{j-1}-u<V_{1}\leq y_{j}-u)}{\sum_{k=1}^{n+1}e^{\theta_{ik}}%
P(y_{k-1}-u<V_{1}\leq y_{k}-u)}\\
\leq &  \max\{~|\psi_{N}^{\prime}(\max\{\theta_{ik},k=1,\ldots,n+1\})|~,~|\psi
_{N}^{\prime}(\min\{\theta_{ik},k=1,\ldots,n+1\})|~\}
\end{align*}
which is uniformly bounded over a neighborhood of $\theta_{ij}$ and
$t_{i-1}<u\leq t_{i}$, fixing all other $\theta_{lk}$'s. Therefore,
\begin{align*}
&  \frac{1}{h}\Bigg|\psi_{N}\left(  \log\left(  e^{\theta_{ij}+h}%
P(y_{j-1}-u<V_{1}\leq y_{j}-u)+\sum_{k\neq j}e^{\theta_{ik}}P(y_{k-1}%
-u<V_{1}\leq y_{k}-u)\right)  \right)  {}\\
&  ~~~~ {}-\psi_{N}\left(  \log\sum_{k=1}^{n+1}e^{\theta_{ik}}P(y_{k-1}%
-u<V_{1}\leq y_{k}-u)\right)  \Bigg|
\end{align*}
is also uniformly bounded one the same region. By dominated convergence
theorem, we have
\begin{align*}
&  ~~\frac{\partial}{\partial\theta_{ij}}\Lambda(\Theta)\\
=  &  \int_{t_{i-1}}^{t_{i}}\psi_{N}^{\prime}\left(  \log\sum_{k=1}%
^{n+1}e^{\theta_{ik}}P(y_{k-1}-u<V_{1}\leq y_{k}-u)\right)  \cdot
\frac{e^{\theta_{ij}}P(y_{j-1}-u<V_{1}\leq y_{j}-u)}{\sum_{k=1}^{n_{i}%
}e^{\theta_{ik}}P(y_{k-1}-u<V_{1}\leq y_{k}-u)}du
\end{align*}
Moreover, it is dominated by
\[
(t_{i}-t_{i-1})\max\{|\psi_{N}^{\prime}(\max\{\theta_{ik},k=1,\ldots
,n+1\})|,|\psi_{N}^{\prime}(\min\{\theta_{ik},k=1,\ldots,n+1\})|\}<\infty
\]
for any given $\Theta\in\mathbb{R}^{d}$. Since $\Lambda(\cdot)$ is finite and
differentiable everywhere on $\mathbb{R}^{d}$, by the Gartner-Ellis Theorem
for the case $\mathcal{D}_{\Lambda}=\mathbb{R}^{d}$ (\cite{DEMZEI98}, p. 52,
Ex 2.3.20 (g)), $\{\Delta_{ij}(\lambda)\}$ possesses a rate function
\[
\sup_{\theta_{ij}:1\leq i\leq m,1\leq j\leq n+1}\sum_{i=1}^{m}\sum_{j=1}%
^{n+1}\theta_{ij}\delta_{ij}-\sum_{i=1}^{m}\int_{t_{i-1}}^{t_{i}}\psi
_{N}\left(  \log\sum_{j=1}^{n+1}e^{\theta_{ij}}P(y_{j-1}-u<V_{1}\leq
y_{j}-u)\right)  du.
\]
We argue that the rate function is good. By \cite{DEMZEI98} p. 8, Lemma
1.2.18, it suffices to show that $(\Delta_{ij}(\lambda):1\leq i\leq m,1\leq
j\leq n+1)$ is exponentially tight. Denoting $\Vert\cdot\Vert_{1}$ as the
$L_{1}$-norm, we have by Chernoff's bound
\[
\overline{\lim}_{\lambda\rightarrow\infty}\frac{1}{\lambda}\log P\left(
\left\Vert \Delta_{ij}(\lambda)/\lambda\right\Vert _{1}>\alpha\right)
\leq\overline{\lim}_{\lambda\rightarrow\infty}\frac{1}{\lambda}\log
P(N_{\lambda}(T)>\alpha\lambda)\leq-\theta\alpha+\psi_{N}(\theta),
\]
for any $\theta>0$. Sending $\alpha\rightarrow\infty$ we then obtain
\[
\overline{\lim}_{\alpha\rightarrow\infty}\overline{\lim}_{\lambda
\rightarrow\infty}\frac{1}{\lambda}\log P\left(  \left\Vert \Delta
_{ij}(\lambda)/\lambda\right\Vert _{1}>\alpha\right)  =-\infty,
\]
thereby obtaining exponential tightness and the goodness of the underlying
rate function as claimed.
\end{proof}

\begin{proof}
[Proof of Lemma \ref{projection limit}]We will use the Dawson-Gartner
projective limit theorem. Consider a collection of points in the plane of the
form $\kappa=\left(  \left(  t_{i},y_{j}\right)  :1\leq i\leq m,0\leq j\leq
n\right)  $, such that $0:=t_{0}<t_{1}<t_{2}<...<t_{m}\leq T$ and
$0:=y_{0}<y_{1}<...<y_{n}$. Moreover, we assume that $y_{l}=t_{l}$ if $0\leq
l\leq\min\left(  m,n\right)  $. Let $K$ be the union of such collection of
sets $\kappa$. Further, let $\{p_{\kappa}\}_{\kappa\in K}$ be the projective
system generated by $K$. We will proceed to obtain a large deviations
principle for the projections $\left(  \bar{Q}_{\lambda}(t,y)/\lambda:\left(
t,y\right)  \in\kappa\right)  $. However, we will do this by first obtaining a
large deviations principle for quantities $\Delta_{ij}\left(  \lambda\right)
/\lambda$ and then the large deviations principle for the projections follows
using the contraction principle as the $\left(  \bar{Q}_{\lambda}%
(t,y)/\lambda:\left(  t,y\right)  \in\kappa\right)  $ will be shown to be
continuous functions. Set $y_{n+1}=\infty$, so that $\bar{Q}_{\lambda
}(t,y_{n+1})=0$ for every $t\in\lbrack0,T]$. It is important to note, given
the structure of the partition $\kappa$, that if $1\leq i\leq m$, $1\leq j\leq
n$, and $i>j$, then $\Delta_{ij}\left(  \lambda\right)  =0$. Now, similar to
the definition of $\Delta_{ij}\left(  \lambda\right)  $ we define, for $1\leq
i\leq m$ and $1\leq j\leq n+1$,
\[
\widetilde{\Delta}_{ij}\left(  \lambda\right)  =\tilde{Q}_{\lambda}%
(t_{i},y_{j-1})-\tilde{Q}_{\lambda}(t_{i},y_{j})-\tilde{Q}_{\lambda}%
(t_{i-1},y_{j-1})+\tilde{Q}_{\lambda}(t_{i-1},y_{j}).
\]
Once again, observe that $\tilde{Q}_{\lambda}(t,y_{n+1})=0$, and also if
$i>j$, for $1\leq i\leq m$, $1\leq j\leq n$, we have $t_{i-1}\geq y_{j}$ and
therefore%
\begin{align*}
\widetilde{\Delta}_{ij}\left(  \lambda\right)   &  =\tilde{Q}_{\lambda
}(y_{j-1},y_{j-1})+\tilde{N}_{\lambda}\left(  t_{i}\right)  -\tilde
{N}_{\lambda}\left(  y_{j-1}\right)  -(\tilde{Q}_{\lambda}(y_{j},y_{j}%
)+\tilde{N}_{\lambda}\left(  t_{i}\right)  -\tilde{N}_{\lambda}\left(
y_{j}\right)  )\\
&  -(\tilde{Q}_{\lambda}(y_{j-1},y_{j-1})+\tilde{N}_{\lambda}\left(
t_{i-1}\right)  -\tilde{N}_{\lambda}\left(  y_{j-1}\right)  )+(\tilde
{Q}_{\lambda}(y_{j},y_{j})+\tilde{N}_{\lambda}\left(  t_{i-1}\right)
-\tilde{N}_{\lambda}\left(  y_{j}\right)  )\\
&  =0.
\end{align*}
Moreover, clearly we have for $1\leq i\leq m,$ and $1\leq j\leq n$
\[
\tilde{Q}_{\lambda}(t_{i},y_{j})=\sum_{l=1}^{i}\sum_{r=j+1}^{n+1}%
\widetilde{\Delta}_{lr}\left(  \lambda\right)  ,
\]
so indeed we have that $(\tilde{Q}_{\lambda}(t_{i},y_{j}):1\leq i\leq m,1\leq
j\leq n+1)$ can be recovered as a continuous function of the $\widetilde
{\Delta}_{lr}\left(  \lambda\right)  $'s. Since $\left\vert \widetilde{\Delta
}_{ij}\left(  \lambda\right)  -\Delta_{ij}\left(  \lambda\right)  \right\vert
\leq16$, we have that
\[
\lim_{\lambda\rightarrow\infty}\frac{1}{\lambda}\log E\exp\left\{  \sum
_{i=1}^{m}\sum_{j=i}^{n+1}\theta_{ij}\Delta_{ij}\left(  \lambda\right)
\right\}  =\lim_{\lambda\rightarrow\infty}\frac{1}{\lambda}\log E\exp\left\{
\sum_{i=1}^{m}\sum_{j=i}^{n+1}\theta_{ij}\widetilde{\Delta}_{ij}\left(
\lambda\right)  \right\}  .
\]
Consequently, from Lemma \ref{LemAuxLDPFinite}, the rate function for the
projections represented by $\kappa$ (these projections are denoted by
$p_{\kappa}(\bar{q})$) can be written as
\begin{align*}
&  I(p_{\kappa}(\bar{q}))\\
&  =\sup_{\{\theta_{ij}:1\leq\imath\leq m,1\leq j\leq n+1\}}\sum_{i=1}^{m}%
\sum_{j=1}^{n}\theta_{ij}\delta_{ij}\left(  \kappa\right)  -\sum_{i=1}^{m}%
\int_{t_{i-1}}^{t_{i}}\psi_{N}\left(  \log\sum_{j=i}^{n+1}e^{\theta_{ij}%
}P(y_{j-1}-u<V_{1}\leq y_{j}-u)\right)  du
\end{align*}
To possess a finite $I(p_{\kappa}(\bar{q}))$, the quantity $\delta_{ij}\left(
\kappa\right)  :=\bar{q}(t_{i},y_{j-1})-\bar{q}(t_{i},y_{j})-\bar{q}%
(t_{i-1},y_{j-1})+\bar{q}(t_{i-1},y_{j})$ must satisfy that
\begin{equation}
\delta_{ij}\left(  \kappa\right)  =0 \label{CondAC_LT}%
\end{equation}
for $i>j$, and $1\leq i\leq m$, $1\leq j\leq n+1$; otherwise, if $\delta
_{ij}\left(  \kappa\right)  \neq0$, the rate function can be made arbitrarily
large by picking $\theta_{ij}=c\times sgn(\delta_{ij}\left(  \kappa\right)  )$
with arbitrarily large constant $c>0$ for $1\leq j<i\leq m$, as
\[
\int_{t_{i-1}}^{t_{i}}\psi_{N}\left(  \log\sum_{j=i}^{n+1}e^{\theta_{ij}%
}P(y_{j-1}-u<V_{1}\leq y_{j}-u)\right)  du
\]
is independent of $\theta_{ij}$'s that have $j<i$. In the representation of
the rate function $I(p_{\kappa}(\bar{q}))$ we have also used the fact that
$\bar{q}(t_{i},y_{j})=\sum_{l\leq i,r>j}\delta_{lr}\left(  \kappa\right)  $,
with $\bar{q}(0,y_{j})=0$, so the relation from the $\delta_{ij}\left(
\kappa\right)  $'s to the $\bar{q}(t_{i},y_{j})$ is a one-to-one, continuous
function, so that the contraction principle (Theorem 4.2.1, \cite{DEMZEI98})
is invoked for the above representation for $I(p_{\kappa}(\bar{q}))$. We want
to show that $\sup_{\kappa\in K}I(p_{\kappa}(q))$ is equal to \eqref{rate},
and hence conclude the proof by Dawson-Gartner Theorem (see Theorem 4.6.1,
\cite{DEMZEI98}). Clearly it suffices to concentrate on functions $\bar{q}$
such that $\bar{q}(t,y)=0$ whenever $t>T$ or $y>t+M$ given that we are
assuming service times bounded by $M$. Note that the constraint
(\ref{CondAC_LT}) implies that for any $\bar{q}$, in order that $I(\bar
{q})<\infty$, we must have absolute continuity throughout $0\leq y\leq t\leq
T$ and, moreover, that%
\[
\partial\bar{q}(t,y)/\partial y\partial t=0
\]
almost everywhere on $0\leq y\leq t\leq T$ (see \cite{DEMZEI98} p. 189). We
now focus on $\bar{q}(t,y)$ that is absolutely continuous on $[0,T]\times
[0,T+M]$ and have $\partial\bar{q}(t,y)/\partial y\partial t=0$ almost
everywhere on $0\leq y\leq t\leq T$. Observe that
\begin{align*}
\delta_{ij}\left(  \kappa\right)   &  =\bar{q}(t_{i},y_{j-1})-\bar{q}%
(t_{i},y_{j})-\bar{q}(t_{i-1},y_{j-1})+\bar{q}(t_{i-1},y_{j})\\
&  =-\int_{t_{i-1}}^{t_{i}}\int_{y_{j-1}}^{y_{j}}\frac{\partial^{2}}{\partial
t\partial y}\bar{q}(t,y)dydt.
\end{align*}
Regarding $\theta(\cdot,\cdot)$ as a step function with jumps at
$0=t_{1}<t_{2}<\cdots<t_{m}\leq T$ and $0\leq y_{0}<y_{1}<\cdots<y_{n}%
<y_{n+1}=T+M$, and denote $\mathcal{S}\left(  D\right)  $ as the set of all
step functions on a given domain $D$. We can write
\begin{align}
&  \sup_{\kappa}I(p_{\kappa}(\bar{q}))\nonumber\\
&  =\sup_{\theta(\cdot,\cdot)\in\mathcal{S}[0,T]\times\lbrack0,T+M]}\left\{
\int_{0}^{T}\int_{t}^{t+M}\theta(t,y)\left(  -\frac{\partial^{2}}{\partial
y\partial t}\bar{q}(t,y)\right)  dydt-\int_{0}^{T}\psi_{N}\left(  \log\int
_{t}^{t+M}e^{\theta(t,y)}dF\left(  y-t\right)  \right)  dt\right\}
\label{expression}%
\end{align}
To show that $\sup_{\kappa}I(p_{\kappa}(\bar{q}))\geq I(\bar{q})$ where
$I(\bar{q})$ is as defined in \eqref{rate}, note first that the set of step
functions $\mathcal{S}\left(  [0,T]\times\lbrack0,T+M]\right)  $ is dense in
$\mathcal{C}\left(  [0,T]\times\lbrack0,T+M]\right)  $, the set of continuous
functions equipped with the uniform metric. So for any continuous function
$\theta(\cdot,\cdot)\in\mathcal{C}\left(  [0,T]\times\lbrack0,T+M]\right)  $,
we can find a sequence $\theta_{k}(\cdot,\cdot)\in\mathcal{S}\left(
[0,T]\times\lbrack0,T+M]\right)  $ with $\Vert\theta_{k}-\theta\Vert
_{\lbrack0,T]\times\lbrack0,T+M]}\rightarrow0$. Note that since $\theta$ is
continuous, it is bounded and so $\theta_{k}$ is also uniformly bounded i.e.
$|\theta_{k}(t,y)|\leq C$ for all $k$ and some $C>0$. Consider
\[
\int_{0}^{T}\int_{t}^{t+M}\theta(t,y)\left(  -\frac{\partial^{2}}{\partial
y\partial t}\bar{q}(t,y)\right)  dydt-\int_{0}^{T}\psi_{N}\left(  \log\int
_{t}^{t+M}e^{\theta(t,y)}dF\left(  y-t\right)  \right)  dt
\]
with $\theta\in\mathcal{C}\left(  [0,T]\times\lbrack0,T+M]\right)  $. We want
to show that this can be approximated by the counterpart in $\theta_{k}%
\in\mathcal{S}\left(  [0,T]\times\lbrack0,T+M]\right)  $. Note that
\[
\int_{0}^{T}\int_{t}^{t+M}\left\vert \theta_{k}(t,y)\left(  -\frac
{\partial^{2}}{\partial y\partial t}\bar{q}(t,y)\right)  \right\vert dydt\leq
C\int_{0}^{T}\int_{t}^{t+M}\left\vert \frac{\partial^{2}}{\partial y\partial
t}\bar{q}(t,y)\right\vert dydt<\infty
\]
since $\bar{q}$ is absolutely continuous. By dominated convergence we have
\begin{equation}
\int_{0}^{T}\int_{t}^{t+M}\theta_{k}(t,y)\left(  -\frac{\partial^{2}}{\partial
y\partial t}\bar{q}(t,y)\right)  dydt\rightarrow\int_{0}^{T}\int_{t}%
^{t+M}\theta(t,y)\left(  -\frac{\partial^{2}}{\partial y\partial t}\bar
{q}(t,y)\right)  dydt \label{convergence1}%
\end{equation}
Similarly, since, as mentioned earlier $|\theta_{k}(t,y)|\leq C$, by the
bounded convergence theorem we have
\[
\int_{t}^{t+M}e^{\theta_{k}(t,y)}dF\left(  y-t\right)  \rightarrow\int
_{t}^{t+M}e^{\theta_{k}(t,y)}dF\left(  y-t\right)
\]
and so by the continuity of $\psi_{N}(\log\left(  \cdot\right)  )$ we get
\[
\psi_{N}\left(  \log\int_{t}^{t+M}e^{\theta_{k}(t,y)}dF\left(  y-t\right)
\right)  \rightarrow\psi_{N}\left(  \log\int_{t}^{t+M}e^{\theta(t,y)}dF\left(
y-t\right)  \right)
\]
for any $t$. Furthermore, the obvious inequality
\begin{equation}
e^{-C}=e^{-C}\int_{0}^{M}dF(y)\leq\int_{0}^{M}e^{\theta_{k}(t,y)}dF(y)\leq
e^{C}\int_{0}^{M}dF(y)=e^{C}, \label{bdd dominated convergence}%
\end{equation}
yields%
\[
\left\vert \psi_{N}\left(  \log\int_{t}^{t+M}e^{\theta_{k}(t,y)}%
f(y-t)dy\right)  \right\vert \leq\sup_{\xi\in\lbrack-C,C]}|\psi_{N}(\xi)|
\]
Hence yet another application of dominated convergence gives
\begin{equation}
\int_{0}^{T}\psi_{N}\left(  \log\int_{t}^{t+M}e^{\theta_{k}(t,y)}%
dF(y-t)\right)  dt\rightarrow\int_{0}^{T}\psi_{N}\left(  \log\int_{t}%
^{t+M}e^{\theta(t,y)}dF(y-t)\right)  dt \label{convergence2}%
\end{equation}
Combining \eqref{convergence1} and \eqref{convergence2} and using the
expression in \eqref{expression}, we conclude that $\sup_{\kappa}I(p_{\kappa
}(\bar{q}))\geq I(\bar{q})$ (note a shift of variable $y$ in \eqref{rate}).
For the other direction, consider
\[
\int_{0}^{T}\int_{t}^{t+M}\theta(t,y)\left(  -\frac{\partial^{2}}{\partial
t\partial y}\bar{q}(t,y)\right)  dydt-\int_{0}^{T}\psi_{N}\left(  \log\int
_{0}^{M}e^{\theta(t,y)}dF(y)\right)  dt
\]
now with $\theta\in\mathcal{S}\left(  [0,T]\times\lbrack0,T+M]\right)  $. Note
that we can find a sequence $\theta_{k}\in\mathcal{C}\left(  [0,T]\times
\lbrack0,T+M]\right)  $ such that $\theta_{k}\rightarrow\theta$ pointwise
almost everywhere and that $\theta_{k}$ is uniformly bounded; this sequence
can be found, for example, by convolving $\theta$ with a sequence of
mollifiers (i.e. smooth kernels with bandwidth that tends to zero as
$k\rightarrow\infty$). Exactly the same argument as above would then yield
$\sup_{\kappa}I(p_{\kappa}(\bar{q}))\leq I(\bar{q})$.
Now, let $\bar{q}\in C_{+}(\mathcal{D}_{M})$ and suppose that $\bar{q}$ is not
absolutely continuous. That is, it is not of bounded total variation in the
sense of \cite{DEMZEI98} p. 189. Then, for every $\gamma>0$ there exists
$t_{1}\left(  \gamma\right)  <~...~<t_{m}\left(  \gamma\right)  $ and
$y_{0}\left(  \gamma\right)  <...<y_{n}\left(  \gamma\right)  $ such that
$\sum_{i=1}^{m}\sum_{j=1}^{n}\left\vert \delta_{ij}^{\gamma}\right\vert
\geq\gamma$, where%
\[
\delta_{ij}^{\gamma}=\bar{q}(t_{i}\left(  \gamma\right)  ,y_{j-1}\left(
\gamma\right)  )-\bar{q}(t_{i}\left(  \gamma\right)  ,y_{j}\left(
\gamma\right)  )-\bar{q}(t_{i-1}\left(  \gamma\right)  ,y_{j-1}\left(
\gamma\right)  )+\bar{q}(t_{i-1}\left(  \gamma\right)  ,y_{j}\left(
\gamma\right)  ).
\]
Now observe that
\begin{align}
&  \sup_{\kappa\in K}I(p_{\kappa}(\bar{q}))\nonumber\\
&  =\sup_{\substack{\theta_{ij}:1\leq i\leq m,1\leq j\leq n\\\kappa\in K}%
}\sum_{i=1}^{m}\sum_{j=1}^{n}\theta_{ij}\delta_{ij}\left(  \kappa\right)
-\sum_{i=1}^{m}\int_{t_{i-1}}^{t_{i}}\psi_{N}\left(  \log\sum_{j=1}%
^{n+1}e^{\theta_{ij}}P(y_{j-1}-u<V_{1}\leq y_{j}-u)\right)  du.\nonumber
\end{align}
Following \cite{DEMZEI98} p. 192, we can select $\theta_{ij}=sgn\left(
\delta_{ij}^{\gamma}\right)  $ for the partition introduced earlier that
defines $\delta_{ij}^{\gamma}$, and obtain
\[
\sup_{\kappa\in K}I(p_{\kappa}(\bar{q}))\geq\sum_{i=1}^{m}\sum_{j=1}%
^{n}\left\vert \delta_{ij}^{\gamma}\right\vert -T\psi_{N}\left(  1\right)  .
\]
Since $\gamma>0$ is arbitrary we conclude that
\[
\sup_{\kappa\in K}I(p_{\kappa}(\bar{q}))=\infty
\]
as required.
%

\end{proof}

\begin{proof}
[Proof of Lemma \ref{tightness}]We want to prove that for any $\eta$,
\[
\lim_{\lambda\rightarrow\infty}\frac{1}{\lambda}\log P\left(  \left\Vert
\frac{\tilde{Q}_{\lambda}(0,0)}{\lambda}\right\Vert _{\mathcal{D}_{M}}%
>\eta\right)  =-\infty
\]
and
\[
\lim_{\delta\rightarrow0}\overline{\lim}_{\lambda\rightarrow\infty}\frac
{1}{\lambda}\log P\left(  w\left(  \frac{\tilde{Q}_{\lambda}}{\lambda}%
,\delta\right)  >\eta\right)  =-\infty
\]
where $w(\tilde{Q}_{\lambda}/\lambda,\delta)$ is the modulus of continuity of
$\tilde{Q}_{\lambda}/\lambda$ with order $\delta$ defined by
\[
w\left(  \tilde{Q}_{\lambda}/\lambda,\delta\right)  =\sup_{\substack{|t_{1}%
-t_{2}|<\delta\\|y_{1}-y_{2}|<\delta}}|\tilde{Q}_{\lambda}(t_{1}%
,y_{1})/\lambda-\tilde{Q}_{\lambda}(t_{2},y_{2})/\lambda|.
\]
Recall that $\Vert\tilde{Q}_{\lambda}-\bar{Q}_{\lambda}\Vert_{\mathcal{D}_{M}%
}\leq4$ a.s., and that $\bar{Q}_{\lambda}(t,y)=Q_{\lambda}(t,y-t)$ for $y>t$
and $\bar{Q}_{\lambda}(t,y)=\bar{Q}_{\lambda}(y,y)+N_{\lambda}\left(
t\right)  -N_{\lambda}\left(  y\right)  $ for $0\leq y\leq t\leq T$.
Therefore, it suffices to show that for any $\eta>0$,
\begin{equation}
\lim_{\lambda\rightarrow\infty}\frac{1}{\lambda}\log P\left(  \left\Vert
\frac{Q_{\lambda}(0,0)}{\lambda}\right\Vert >\eta\right)  =-\infty,
\label{boundedness}%
\end{equation}
(where $\Vert\cdot\Vert$ denotes the supremum norm over $[0,T]\times
\lbrack0,M]$), also
\begin{equation}
\lim_{\delta\rightarrow0}\overline{\lim}_{\lambda\rightarrow\infty}\frac
{1}{\lambda}\log P\left(  w\left(  \frac{Q_{\lambda}}{\lambda},\delta\right)
>\eta\right)  =-\infty, \label{equicontinutiy}%
\end{equation}
and finally that
\begin{equation}
\lim_{\delta\rightarrow0}\overline{\lim}_{\lambda\rightarrow\infty}\frac
{1}{\lambda}\log P\left(  \sup_{0\leq t_{2}-t_{1}<\delta}(N_{\lambda}%
(t_{2})/\lambda-N_{\lambda}(t_{1})/\lambda)>\eta\right)  =-\infty.
\label{LastAux}%
\end{equation}
By our assumption that the system is empty, \eqref{boundedness} is obvious.
Condition (\ref{LastAux}) will follow as a direct consequence of our analysis
of \eqref{equicontinutiy}. Now, to prove \eqref{equicontinutiy} consider
\[
P\left(  w\left(  \frac{Q_{\lambda}}{\lambda},\delta\right)  >\eta\right)
\leq\sum_{m=0}^{\lfloor T/\delta\rfloor}\sum_{n=0}^{\lfloor M/\delta\rfloor
}P\left(  \sup_{\substack{0<t_{1}-t_{2}<\delta,\ t_{1}\in(m\delta
,(m+1)\delta]\\|y_{2}-y_{1}|<\delta,\ y_{1}\in(n\delta,(n+1)\delta
]}}|Q_{\lambda}(t_{1},y_{1})-Q_{\lambda}(t_{2},y_{2})|>\lambda\eta\right)
\]
\begin{figure}[th]
\centering
\subfloat[][Representation of $Q(t,y)$]{
		\includegraphics[scale=0.3]{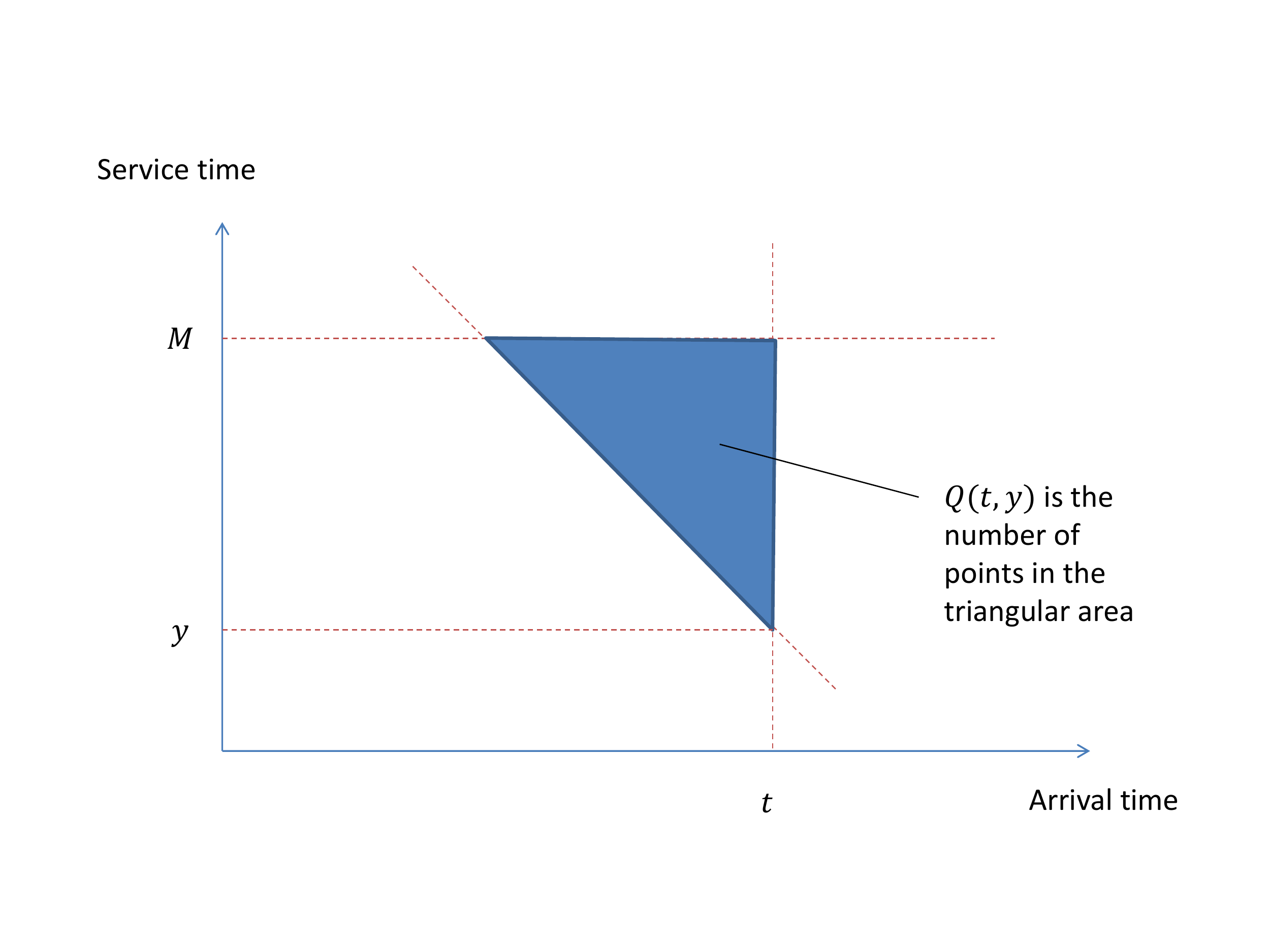}
		\label{triangle}
}
\subfloat[][Areas of $\mathcal{N}_1^\delta(m,n)$ and $\mathcal{N}_2^\delta(m,n)$]{
\includegraphics[scale=0.3]{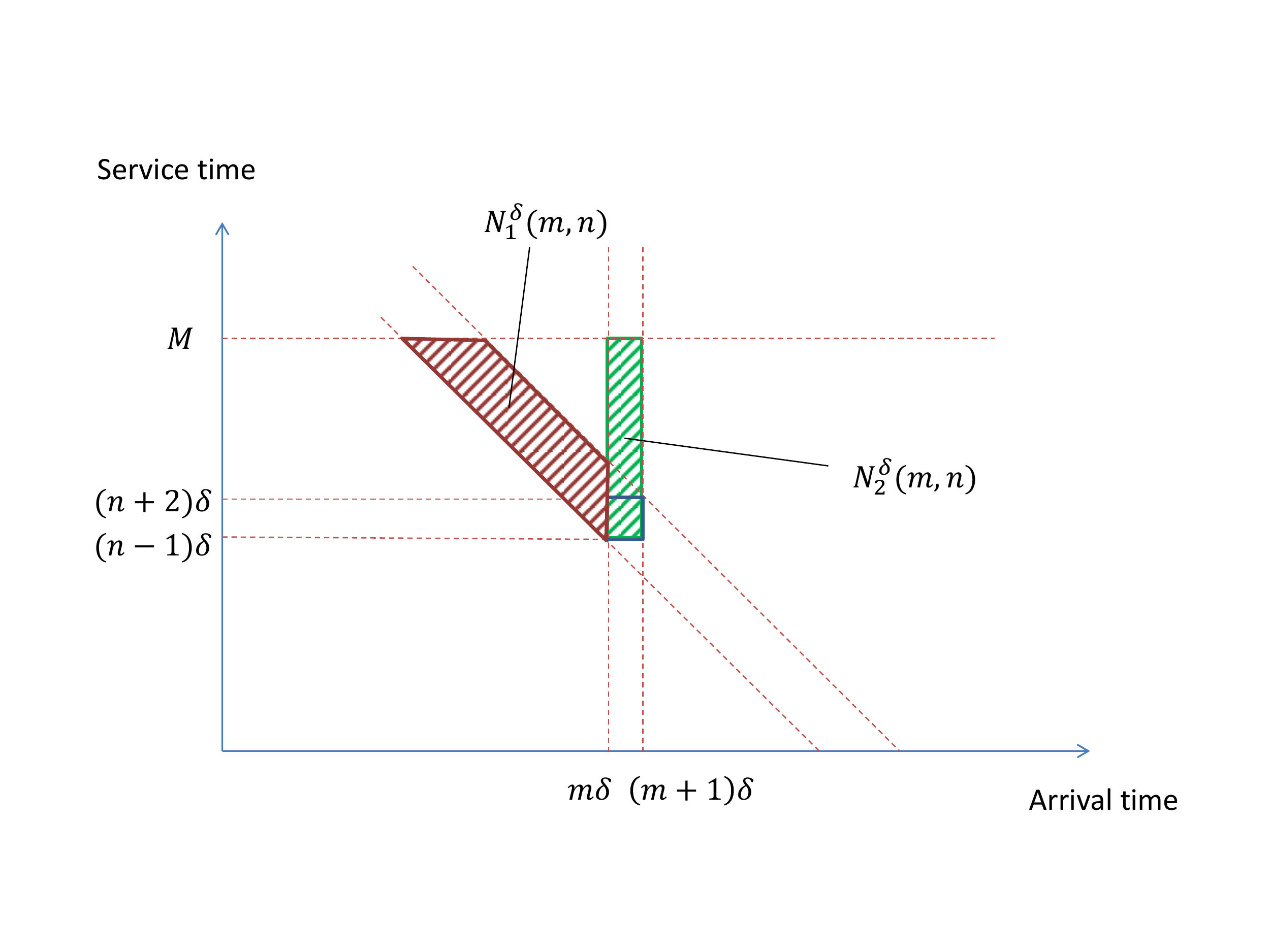}
		\label{N12}
}\caption{Illustrations for $Q_{\lambda}(t,y)$}%
\end{figure}It is best to proceed our analysis by keeping in mind the
pictorial representation that we shall describe. One can represent the arrival
and status of each customer in a two-dimensional plane, with $x$-axis
representing the arrival time and $y$-axis the service time at the time of
arrival. Under this representation, $Q_{\lambda}(t,y)$ is the number of points
in the triangle formed by a vertical line and a $45^{\circ}$ line passing
through $(t,y)$ in its northwest direction. Figure \ref{triangle} depicts the
shape of this triangle. Consequently, we have
\begin{align*}
&  P\left(  \sup_{\substack{0<t_{1}-t_{2}<\delta,\ t_{1}\in(m\delta
,(m+1)\delta],\\|y_{2}-y_{1}|<\delta,\ y_{1}\in(n\delta,(n+1)\delta
]}}|Q_{\lambda}(t_{1},y_{1})-Q_{\lambda}(t_{2},y_{2})|>\lambda\eta\right) \\
&  \leq P(\mathcal{N}_{1}^{\delta}(m,n,\lambda)+\mathcal{N}_{2}^{\delta
}(m,n,\lambda)>\eta\lambda)
\end{align*}
where
\[
\mathcal{N}_{1}^{\delta}(m,n,\lambda)=Q_{\lambda}(m\delta,(n-1)\delta
)-Q_{\lambda}(m\delta,(n+2)\delta)
\]
is the number of customers present at time $m\delta$ who have residual service
time between $(n-1)\delta$ and $(n+2)\delta$, and
\[
\mathcal{N}_{2}^{\delta}(m,n,\lambda)=\sum_{i=N_{\lambda}(m\delta
)+1}^{N_{\lambda}((m+1)\delta)}I(V_{i}>(n-1)\delta),
\]
is the number of arrivals between $m\delta$ and $(m+1)\delta$ which bring
service requirements larger than $(n-1)\delta$. Figure \ref{N12} depicts the
areas under which the points are included in $\mathcal{N}_{1}^{\delta
}(m,n,\lambda)$ and $\mathcal{N}_{2}^{\delta}(m,n,\lambda)$. Fixing $\delta>0$
and $\theta>0$ we take the limit as $\lambda\rightarrow\infty$ in the
following display, obtaining
\begin{align}
&  \frac{1}{\lambda}\log Ee^{\theta(\mathcal{N}_{1}^{\delta}(m,n,\lambda
)+\mathcal{N}_{2}^{\delta}(m,n,\lambda))}\nonumber\\
=~  &  \frac{1}{\lambda}\log E\exp\Bigg\{\theta\Bigg(\sum_{i=1}^{N_{\lambda
}(m\delta)}I(m\delta+(n-1)\delta-A_{i}/\lambda<V_{i}\leq m\delta
+(n+2)\delta-A_{i}/\lambda)\nonumber\\
&  ~~+\sum_{i=N_{\lambda}(m\delta)+1}^{N_{\lambda}((m+1)\delta)}%
I(V_{i}>(n-1)\delta)\Bigg)\Bigg\}\nonumber\\
=~  &  \frac{1}{\lambda}\log E\exp\Bigg\{\int_{0}^{m\delta}\log(e^{\theta
}P(m\delta+(n-1)\delta-u<V_{i}\leq m\delta+(n+2)\delta-u)+1\nonumber\\
&  ~~-P(m\delta+(n-1)\delta-u<V_{i}\leq m\delta+(n+2)\delta-u))dN_{\lambda
}(u)+\log(e^{\theta}\bar{F}((n-1)\delta)\nonumber\\
&  ~~+F((n-1)\delta))[N_{\lambda}((m+1)\delta)-N_{\lambda}(m\delta
)]\Bigg\}\nonumber\\
\rightarrow &  \int_{0}^{m\delta}\psi_{N}(\log(e^{\theta}P(m\delta
+(n-1)\delta-u<V_{i}\leq m\delta+(n+2)\delta-u)+1\nonumber\\
&  ~~-P(m\delta+(n-1)\delta-u<V_{i}\leq m\delta+(n+2)\delta-u)))du\nonumber\\
&  ~~+\psi_{N}(\log(e^{\theta}\bar{F}((n-1)\delta)+F((n-1)\delta
)))\delta\nonumber\label{uniform convergence}\\
=:  &  ~\psi_{\delta}(\theta;m,n)\nonumber
\end{align}
Fix $\theta\geq0$. We argue that $\psi_{\delta}(\theta,m,n)\rightarrow0$ as
$\delta\rightarrow0$ uniformly over $m,n$. Indeed, for any $m,n$,
\begin{align*}
&  \int_{0}^{m\delta}\psi_{N}(\log(e^{\theta}P(m\delta+(n-1)\delta-u<V_{i}\leq
m\delta+(n+2)\delta-u)+1\\
&  ~~-P(m\delta+(n-1)\delta-u<V_{i}\leq m\delta+(n+2)\delta-u)))du\\
\leq &  \int_{0}^{M}\psi_{N}(\log(e^{\theta}P((n-1)\delta+u<V_{i}%
\leq(n+2)\delta+u)+1-P((n-1)\delta+u<V_{i}\leq(n+2)\delta+u)))du\\
\leq &  ~M\psi_{N}(\log(e^{\theta}\alpha(\delta)+1-\alpha(\delta)))
\end{align*}
where $\alpha(\delta):=\sup_{x\in\lbrack0,M]}P(x<V_{i}\leq x+3\delta)=o(1)$ as
$\delta\rightarrow\infty$ by our assumption that the distribution of $V_{i}$
is continuous and the fact that a continuous function is uniformly continuous
on a compact set. On the other hand, we also have
\[
\psi_{N}(\log(e^{\theta}\bar{F}((n-1)\delta)+F((n-1)\delta)))\delta\leq
\psi_{N}(\theta)\delta
\]
for any $m,n$. Combining, we get
\begin{equation}
\psi_{\delta}(\theta,m,n)\leq M\psi_{N}(\log(e^{\theta}\alpha(\delta
)+1-\alpha(\delta)))+\psi_{N}(\theta)\delta\label{GE}%
\end{equation}
Now fix $m$ and $n$. By Chernoff's inequality we get
\[
P(\mathcal{N}_{1}^{\delta}(m,n,\lambda)+\mathcal{N}_{2}^{\delta}%
(m,n,\lambda)>\eta\lambda)\leq e^{-\eta\theta\lambda+\psi_{\delta}%
(\theta,m,n)\lambda+o(\lambda)}%
\]
and so
\begin{align*}
&  \overline{\lim}_{\lambda\rightarrow\infty}\frac{1}{\lambda}\log P\left(
\mathcal{N}_{1}(m,n,\lambda)+\mathcal{N}_{2}(m,n,\lambda)>\eta\lambda\right)
\\
\leq~  &  -\eta\theta+\psi_{\delta}(\theta,m,n)\\
\leq~  &  -\eta\theta+M\psi_{N}(\log(e^{\theta}\alpha(\delta)+1-\alpha
(\delta)))+\psi_{N}(\theta)\delta
\end{align*}
by \eqref{GE}.
Hence
\begin{align*}
&  \overline{\lim}_{\lambda\rightarrow\infty}\frac{1}{\lambda}\log P\left(
w\left(  \frac{Q_{\lambda}}{\lambda},\delta\right)  >\eta\right) \\
\leq~  &  \overline{\lim}_{\lambda\rightarrow\infty}\frac{1}{\lambda}\log
\sum_{m=0}^{\lfloor T/\delta\rfloor}\sum_{n=0}^{\lfloor M/\delta\rfloor
}P(\mathcal{N}_{1}^{\delta}(m,n,\lambda)+\mathcal{N}_{2}^{\delta}%
(m,n,\lambda)>\eta\lambda)\\
\leq~  &  -\eta\theta+M\psi_{N}(\log(e^{\theta}\alpha(\delta)+1-\alpha
(\delta)))+\psi_{N}(\theta)\delta
\end{align*}
which gives
\[
\lim_{\delta\rightarrow0}\limsup_{\lambda\rightarrow\infty}\frac{1}{\lambda
}\log P\left(  w\left(  \frac{Q_{\lambda}}{\lambda},\delta\right)
>\eta\right)  \leq-\eta\theta
\]
Since $\theta$ can be arbitrarily large, we conclude \eqref{equicontinutiy}.
Finally, condition (\ref{LastAux}) follows from the analysis of $\mathcal{N}%
_{2}^{\delta}(m,1,\lambda)/\lambda$.

\end{proof}

\section{Unbounded Service Times\label{SectionUnBdd}}

In this section, we will extend our result to unbounded service times. The
main intuition of the extension beyond the bounded case is to justify that we
can ignore in certain sense the customers who arrive with very large service
time. Let us first introduce a suitable truncation scheme. For any $K>0$ and
$\bar{q}\in AC_{+}\left(  \mathcal{D}\right)  $ define%
\begin{equation}
\phi_{K}(\bar{q})(t,y)=\int_{0}^{t}\int_{y-w}^{K}-\left.  \frac{\partial^{2}%
}{\partial s\partial z}\bar{q}(s,z)\right\vert _{s=w ,z=r+w}drdw
\label{phi K q}%
\end{equation}
for $t\in\lbrack0,T]$ and $y:=u+t\geq0$.

\begin{figure}[th]
\centering
\includegraphics[scale=0.3]{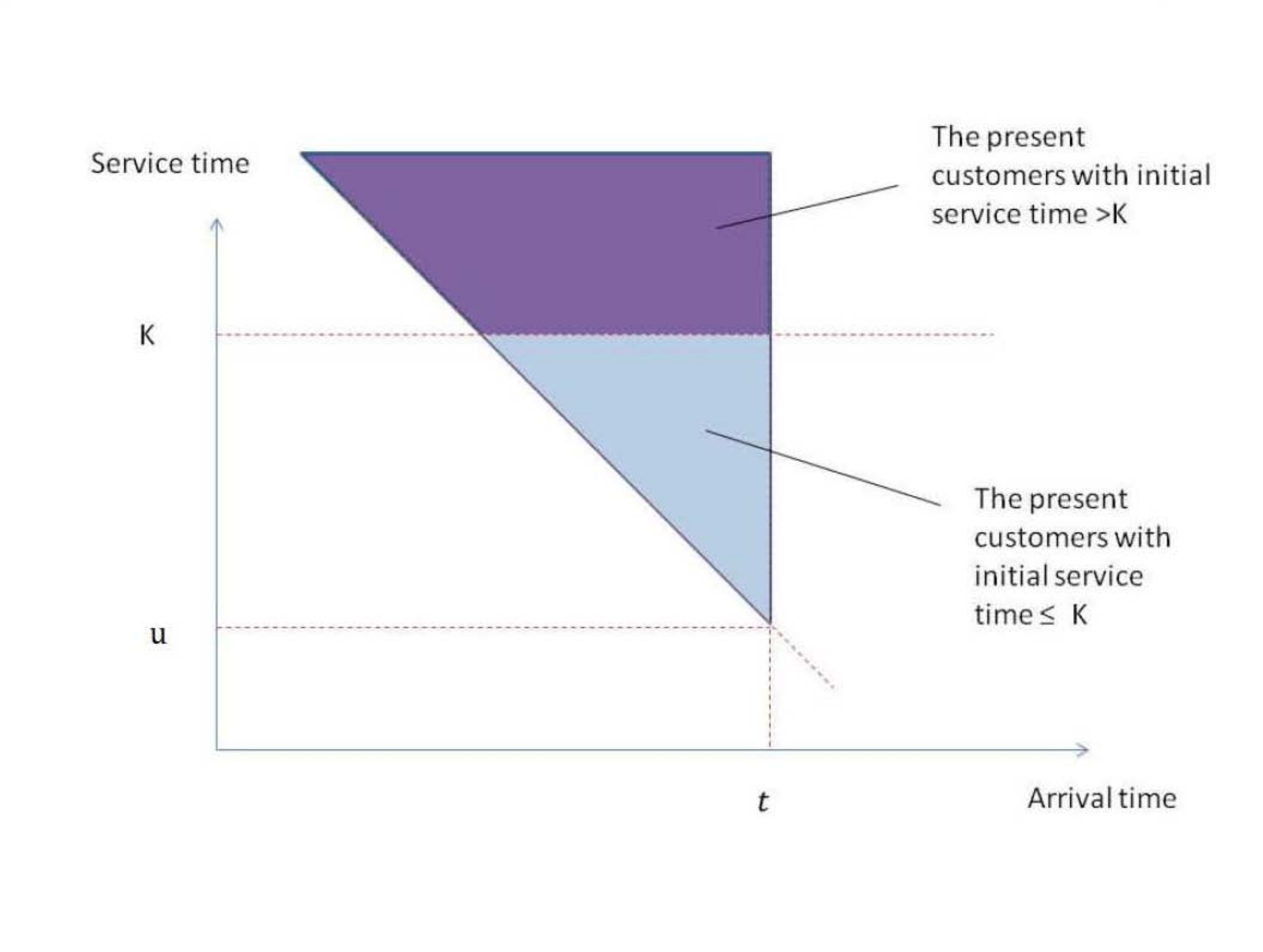}
\caption{Illustration for $\phi_{K}(\bar{q})(t,t+u)$}%
\label{fig3}%
\end{figure}


Since $\bar{q}$ is absolutely continuous, $\phi_{K}\left(  \bar{q}\right)
\left(  t,u+t\right)  $ is well defined. Moreover, the region over which the
integration in \eqref{phi K q} is performed corresponds to the triangular area
depicted in light color in Figure \ref{fig3}. This region corresponds to the
customers that are present at time $t$, have residual residual service time
greater than $y$, and whose initial service time is less than $K$, as
illustrated in Figure \ref{fig3}.

Moreover, for a sample path $\bar{Q}_{\lambda}$, define $\bar{Q}_{\lambda,K}$
as the two-parameter process derived from $\bar{Q}_{\lambda}$ by ignoring the
arrivals with service time greater than $K$ (one way to imagine is that they
leave the system immediately upon arrival). Therefore, $\bar{Q}_{\lambda,K}$
is a two-parameter queue length process corresponding to an infinite server
system with i.i.d. interarrival times following the law $\bar{U}=\sum
_{i=1}^{G}U_{i}/\lambda$, where $G$ is a geometric r.v. independent of the
$U_{i}$'s such that $P\left(  G=n\right)  =\bar{F}(K)^{n-1}F(K)$, $n\geq1$. It
is easy to check that the arrival process corresponding to $\bar{Q}%
_{\lambda,K}$, i.e. by ignoring the arrivals with initial service time larger
than $K$, satisfies the conditions in Section 2.1. The service time then has
the distribution function $F_{K}(x)=F(x)/F(K)$ for $x\in\lbrack0,K]$. We
denote $(V_{n}^{(K)},n=1,\ldots)$ as the sequence of service times in this
modified system.

Now recall the continuous version of $\bar{Q}_{\lambda}$, denoted by
$\tilde{Q}_{\lambda}$ constructed in Section \ref{SubSectAuxiliary}. Moreover,
define $\tilde{Q}_{\lambda,K}$ to be the continuous approximation to $\bar
{Q}_{\lambda,K}$ constructed in exactly the same fashion. In addition, for
$\bar{q}\in AC_{+}\left(  \mathcal{D}_{K}\right)  $ define $I_{K}\left(
\bar{q}\right)  $ as%
\begin{equation}
\sup_{\theta(t,\cdot)\in C[0,T]\times\lbrack0,T+K]}\int_{0}^{T}\left[
\int_{t}^{t+K}\theta(t,y-t)\left(  -\frac{\partial^{2}}{\partial t\ \partial
y}\bar{q}(t,y)\right)  dy-\psi_{N}^{(K)}\left(  \log\left(  \frac{1}{F(K)}%
\int_{0}^{K}e^{\theta(t,y)}dF(y)\right)  \right)  \right]  dt, \label{RT}%
\end{equation}
and set $I_{K}\left(  \bar{q}\right)  =\infty$ otherwise, where $\psi
_{N}^{(K)}$ is the infinitesimal moment generating function corresponding to
the truncated arrival process.

Theorem \ref{ThmBdd1} yields that $\tilde{Q}_{\lambda,K}/\lambda$, satisfies a
full large deviations principle with good rate function $I_{K}(\cdot)$. For
$\bar{q}\in AC_{+}\left(  \mathcal{D}\right)  $ we shall also evaluate
$I_{K}\left(  \bar{q}\right)  $ according to the expression (\ref{RT}).

Since the geometric r.v. $G$ is independent of the $U_{i}$'s, we can compute
the associated logarithmic moment generating function of the modified
interarrival times%
\[
\kappa^{(K)}(\theta):=\kappa(\theta)+\log\left(  \frac{F(K)}{1-\bar
{F}(K)e^{\kappa(\theta)}}\right)  ,
\]
and from which we solve that the associated infinitesimal logarithmic moment
generating function of the arrival process is%
\[
\psi_{N}^{(K)}(\theta):=\psi_{N}(\log(F(K)e^{\theta}+\bar{F}(K))).
\]

Plugging in the above expressions into \eqref{RT}, we have the following
expression of $I_{K}(\bar{q})$
\begin{equation}
\sup_{\theta(\cdot,\cdot)\in C[0,T]\times\lbrack0,T+K]}\int_{0}^{T}\left[
\int_{t}^{T+K}\theta(t,y-t)\left(  -\frac{\partial^{2}}{\partial t\ \partial
y}\bar{q}(t,y)\right)  dy-\psi_{N}\left(  \log\left(  \bar{F}(K)+\int_{0}%
^{K}e^{\theta(t,y)}dF(y)\right)  \right)  \right]  dt. \label{rate unbdd}%
\end{equation}

At this point our strategy involves two steps. First, we want to show that
$\bar{Q}_{\lambda,K}/\lambda$ and $\tilde{Q}_{\lambda,K}/\lambda$ are
exponentially good approximations as $K\nearrow\infty$ to both $\bar
{Q}_{\lambda}/\lambda$ and $\tilde{Q}_{\lambda}/\lambda$ respectively. The
second step consists in using this fact, together with the properties of
$I_{K}(\bar{q})$ as $K\nearrow\infty$ and also properties of $I\left(  \bar
{q}\right)  $ to conclude the identification of the rate function of
$\tilde{Q}_{\lambda}/\lambda$.

So, to execute the first step we first define
\[
N_{\lambda}^{\left(  K\right)  }\left(  t\right)  =\sum_{j=1}^{N_{\lambda
}\left(  t\right)  }I\left(  V_{j}>K\right)  ,
\]
that is, $N_{\lambda}^{\left(  K\right)  }\left(  t\right)  $ is the number of
arrivals with service time larger than $K$ in the $\lambda$-scaled system.
Then we obtain the following result, which is proved at the end of this section.

\begin{lemma}
\label{arrivalsK}For any $\varepsilon>0$,%
\begin{equation}
\lim_{K\rightarrow\infty}\overline{\lim}_{\lambda\rightarrow\infty}\frac
{1}{\lambda}\log P(N_{\lambda}^{\left(  K\right)  }\left(  T\right)
>\lambda\varepsilon)=-\infty. \label{LimNkl1}%
\end{equation}
Consequently, $\bar{Q}_{\lambda,K}/\lambda$ and $\tilde{Q}_{\lambda,K}%
/\lambda$ are exponentially good approximations as $K\nearrow\infty$ to both
$\bar{Q}_{\lambda}/\lambda$ and $\tilde{Q}_{\lambda}/\lambda$ respectively.
\end{lemma}

\bigskip

Using the previous lemma we obtain the following result. The proof is
straightforward, but following our convention we shall give it at the end of
the section.

\begin{lemma}
\label{LemmaWeakLDPUBD}The family $(\tilde{Q}_{\lambda}/\lambda:\lambda>0)$
satisfies a weak large deviations principle on $C_{+}\left(  \mathcal{D}%
\right)  $ with rate function%
\[
I^{\ast}\left(  \bar{q}\right)  :=\sup_{\delta>0}\underline{\lim
}_{K\rightarrow\infty}\inf_{\{z:\left\vert \left\vert z-\bar{q}\right\vert
\right\vert _{\mathcal{D}}\leq\delta\}}I_{K}\left(  z\right)  .
\]

\end{lemma}

\bigskip

We now extend the weak large deviations principle into a full large deviations
principle with a good rate function using exponential tightness.

\begin{lemma}
\label{LemmaLDPUB1}The family $(\tilde{Q}_{\lambda}/\lambda:\lambda>0)$ is
exponentially tight on $C_{+}\left(  \mathcal{D}\right)  $ and therefore it
satisfies a full large deviations principle with good rate function $I^{\ast
}\left(  \cdot\right)  .$\bigskip
\end{lemma}

We proceed to show the identification $I^{\ast}\left(  \bar{q}\right)
=I\left(  \bar{q}\right)  $. We now collect useful properties that we will
need to show this identification.

\begin{lemma}
\label{Iqbar}
\end{lemma}

\textit{i) For any }$\bar{q}$\textit{ such that }$I(\bar{q})<\infty$\textit{,
we have }$I(\phi_{K}(\bar{q}))=I_{K}(\phi_{K}(\bar{q}))=I_{K}(\bar{q})\nearrow
I(\bar{q})$\textit{ as }$K\rightarrow\infty$\textit{; the notation }%
$I_{K}(\bar{q})\nearrow I(\bar{q})$\textit{ implies that }$(I_{K}(\bar
{q}):K>0)$\textit{ is non decreasing in }$K$\textit{ and convergent to
}$I(\bar{q})$\textit{.}

\textit{ii) For any }$\bar{q}$\textit{ such that }$I(\bar{q})=\infty$\textit{,
and each }$M>0$\textit{, there exists a projection }$p_{\kappa}$\textit{
(following the notation introduced in the proof of Lemma
\ref{projection limit}) such that, for large enough $K$, }%
\[
I_{K}(p_{\kappa}(\bar{q}))>M.
\]

\textit{iii) Finally, with }$\kappa$\textit{ from ii) there exists
}$\varepsilon>0$\textit{ such that if }$\hat{q}\in C_{+}\left(  \mathcal{D}%
\right)  $\textit{ and }$\left\vert \left\vert \hat{q}-\bar{q}\right\vert
\right\vert _{\mathcal{D}}<\varepsilon$\textit{ then}%
\[
I_{K}(p_{\kappa}(\hat{q}))>M.
\]

\bigskip

We now are ready to prove the following important result of this section.

\begin{thm}
\label{ThmUnBdd1} $\tilde{Q}_{\lambda}/\lambda$ satisfies a large deviations
principle with good rate function defined in (\ref{R}) under the uniform
topology on $[0,T]\times\lbrack0,\infty)$.
\end{thm}

\begin{proof}
[Proof of Theorem \ref{ThmUnBdd1}]Given Lemma \ref{LemmaLDPUB1} all we need to
show is that $I^{\ast}\left(  \bar{q}\right)  =I\left(  \bar{q}\right)  $.
Suppose that $\bar{q}$ is such that $I\left(  \bar{q}\right)  =\infty$. Then,
parts ii) and iii) in Lemma \ref{Iqbar} imply in particular that for every
$M$, there exists $K$, a projection $\kappa$, and $\varepsilon>0$ such that
$I_{K}\left(  p_{\kappa}\left(  \hat{q}\right)  \right)  >M$ for any
$\left\vert \left\vert \hat{q}-\bar{q}\right\vert \right\vert _{\mathcal{D}%
}<\varepsilon$. Consequently, we conclude, by using the monotonicity of
$I_{K}\left(  \bar{q}\right)  $ as a function of $K$ and taking subsequences,
that%
\[
I^{\ast}\left(  \bar{q}\right)  =\sup_{\delta>0}\underline{\lim}%
_{K\rightarrow\infty}\inf_{\{z:\left\vert \left\vert z-\bar{q}\right\vert
\right\vert _{\mathcal{D}}\leq\delta\}}I_{K}\left(  z\right)  =\sup_{\delta
>0}\sup_{K>0}\inf_{\{z:\left\vert \left\vert z-\bar{q}\right\vert \right\vert
_{\mathcal{D}}\leq\delta\}}I_{K}\left(  z\right)  =\infty.
\]
If $I\left(  \bar{q}\right)  <\infty$, then note that
\[
I^{\ast}\left(  \bar{q}\right)  =\sup_{\delta>0}\sup_{K>0}\inf_{\{z:\left\vert
\left\vert z-\bar{q}\right\vert \right\vert _{\mathcal{D}}\leq\delta\}}%
I_{K}\left(  z\right)  =\sup_{K>0}\sup_{\delta>0}\inf_{\{z:\left\vert
\left\vert z-\bar{q}\right\vert \right\vert _{\mathcal{D}}\leq\delta\}}%
I_{K}\left(  z\right)  .
\]
Further, observe that%
\[
\sup_{\delta>0}\inf_{\{z:\left\vert \left\vert z-\bar{q}\right\vert
\right\vert _{\mathcal{D}}\leq\delta\}}I_{K}\left(  z\right)  =\sup_{\delta
>0}\inf_{\{\phi_{K}\left(  z\right)  :\left\vert \left\vert \phi_{K}\left(
z\right)  -\phi_{K}\left(  \bar{q}\right)  \right\vert \right\vert
_{\mathcal{D}}\leq\delta\}}I_{K}\left(  \phi_{K}\left(  z\right)  \right)  ,
\]
Since $I_{K}\left(  \cdot\right)  $ is a rate function (in particular
$I_{K}\left(  \cdot\right)  $ is lower semicontinuous) we have that%
\[
\sup_{\delta>0}\inf_{\{\phi_{K}\left(  z\right)  :\left\vert \left\vert
\phi_{K}\left(  z\right)  -\phi_{K}\left(  \bar{q}\right)  \right\vert
\right\vert _{\mathcal{D}_{K}}\leq\delta\}}I_{K}\left(  \phi_{K}\left(
z\right)  \right)  =I_{K}\left(  \phi_{K}\left(  \bar{q}\right)  \right)
\]
and then by part i) of Lemma \ref{Iqbar} we conclude that $\sup_{K>0}%
I_{K}\left(  \phi_{K}\left(  \bar{q}\right)  \right)  =I\left(  \bar
{q}\right)  $, thus concluding that $I^{\ast}\left(  \bar{q}\right)  =I\left(
\bar{q}\right)  $ as claimed.
\end{proof}

We finish this section with the proof of Theorem \ref{main thm unbdd}.

\begin{proof}
[Proof of Theorem \ref{main thm unbdd}]All we need to show is that $\bar
{Q}_{\lambda}/\lambda$ and $\tilde{Q}_{\lambda}/\lambda$ are exponentially
equivalent. This follows exactly as in the proof of Corollary \ref{CorMainBdd}
since $\Vert\bar{Q}_{\lambda}-\tilde{Q}_{\lambda}\Vert_{\mathcal{D}}\leq4$.
The measurability issue again is dealt with using separability. The result
then follows by applying Theorem 4.2.13 in \cite{DEMZEI98}.
\end{proof}

\bigskip

\subsection{Proofs of Technical Results}

We now provide the proof of the pending technical results.

\begin{proof}
[Proof of Lemma \ref{LemmaWeakLDPUBD}]This result is a direct application of
part a)\ in Theorem 4.2.16 in \cite{DEMZEI98}.
\end{proof}

\begin{proof}
[Proof of Lemma \ref{LemmaLDPUB1}]This is similar to the case with bounded
service time, but the conditions for tightness are slightly different given
that our domain $\mathcal{D}$ is not compact. We must show that for any
$\eta,\gamma>0$, we can choose small enough $\rho>0$, such that for
$\delta<\rho$,
\[
\frac{1}{\lambda}\log P(w(\tilde{Q}_{\lambda}/\lambda,\delta)>\eta)<-\gamma
\]
when $\lambda$ is large; this part is indeed basically the same as the case
$\mathcal{D}_{M}$. In addition, however, we also must show that for all
$\eta>0$ and every $a>0$ there exists $K>0$ such that%
\begin{equation}
P\left(  \sup_{t\in\lbrack0,T]}\sup_{y\geq K}\tilde{Q}_{\lambda}\left(
t,y\right)  /\lambda>\eta\right)  \leq\exp\left(  -\lambda a\right)  .
\label{CondTight1}%
\end{equation}
Note that
\begin{align}
&  P(w(\tilde{Q}_{\lambda}/\lambda,\delta)>\eta)\nonumber\\
\leq &  P\left(  w(\tilde{Q}_{\lambda}/\lambda,\delta)>\eta,\Vert\tilde
{Q}_{\lambda}/\lambda-\tilde{Q}_{\lambda,K}/\lambda\Vert\leq\frac{\eta}%
{2}\right)  +P\left(  \Vert\tilde{Q}_{\lambda}/\lambda-\tilde{Q}_{\lambda
,K}/\lambda\Vert>\frac{\eta}{2}\lambda\right) \nonumber\\
\leq &  P\left(  w(\tilde{Q}_{\lambda,K}/\lambda,\delta)>\frac{\eta}%
{2}\right)  +P\left(  N_{\lambda}^{\left(  K\right)  }\left(  T\right)
>\lambda\eta/2\right)  , \label{CT2}%
\end{align}
and
\[
P\left(  \sup_{t\in\lbrack0,T]}\sup_{y\geq K}\tilde{Q}_{\lambda}\left(
t,y\right)  /\lambda>\eta\right)  \leq P\left(  N_{\lambda}^{\left(  K\right)
}\left(  T\right)  >\lambda\eta\right)  .
\]
By Lemma \ref{arrivalsK}, for every $\gamma>0$ can choose $K$ large enough
such that
\[
\frac{1}{\lambda}\log P\left(  N_{\lambda}^{\left(  K\right)  }\left(
T\right)  >\lambda\eta/2\right)  <-2\gamma.
\]
for all $\lambda$ large enough. So, condition (\ref{CondTight1}) is enforced
and the second term in the sum in (\ref{CT2}) is also appropriately
controlled. Now, by a similar argument as in Lemma \ref{tightness} in the
previous section, for a chosen $K$, we have, for all small enough $\delta$,
\[
\frac{1}{\lambda}\log P\left(  w(\tilde{Q}_{\lambda,K}/\lambda,\delta
)>\frac{\eta}{2}\right)  <-2\gamma.
\]
for large enough $\lambda$. In sum, we get
\[
\frac{1}{\lambda}\log P(w(\tilde{Q}_{\lambda}/\lambda,\delta)>\eta)<-\gamma
\]
for large enough $\lambda$. Therefore, exponential tightness follows. It
follows immediately that a weak large deviations principle and exponential
tightness implies a full large deviations principle. The goodness of the rate
function then is a consequence of exponential tightness together with the weak
large deviations principle; see Lemma 1.2.18, p. 8, part b) of \cite{DEMZEI98}.
\end{proof}

\begin{proof}
[Proof of Lemma \ref{Iqbar}]We start with part i), assuming that $I(\bar
{q})<\infty$. Since
\[
\frac{\partial^{2}}{\partial t\partial y}\phi_{K}(\bar{q})(t,y)=\frac
{\partial^{2}}{\partial y\partial t}\phi_{K}(\bar{q})(t,y)=\frac{\partial^{2}%
}{\partial t\partial y}\bar{q}(t,y)1_{\{y\leq K+t\}},
\]
we have immediately that $I(\phi_{K}(\bar{q}))=I_{K}(\phi_{K}(\bar{q}%
))=I_{K}(\bar{q})$. It is obvious that $I_{K}(\bar{q})$ is non-decreasing in
$K$ and that $I_{K}\left(  \bar{q}\right)  \leq I(\bar{q})$. On the other
hand, there exists $\theta^{n}\in C_{b}[0,T]\times\lbrack0,\infty)$ (the space
of bounded and continuous functions on $[0,T]\times\lbrack0,\infty)$) such
that
\[
I^{n}(\bar{q}):=\int_{0}^{T}\left[  \int_{t}^{\infty}\theta^{n}(t,y-t)\left(
-\frac{\partial^{2}}{\partial t\ \partial y}\bar{q}(t,y)\right)  dy-\psi
_{N}\left(  \log\left(  \int_{0}^{\infty}e^{\theta^{n}(t,y)}dF(y)\right)
\right)  \right]  dt
\]
converges to $I(\bar{q})$ as $n\rightarrow\infty$. Since $I(\bar{q})<\infty$,
it follows easily that $\partial^{2}\bar{q}(\cdot)/\partial t\partial y$ is
integrable over $\mathcal{D}$. Therefore, given that $\theta^{n}(\cdot)$ is
bounded,%
\[
\int_{0}^{T}\int_{K}^{\infty}\theta^{n}(t,y-t)\left(  -\frac{\partial^{2}%
}{\partial t\ \partial y}\bar{q}(t,y)\right)  dydt\rightarrow0
\]
as $K\rightarrow\infty$. Besides, for given $n$, as $\psi_{N}$ is uniformly
continuous on the bounded set $[-M_{n},M_{n}]$ where $M^{n}=\sup|\theta
^{n}(t,y)|<\infty$, we have that
\[
\psi_{N}\left(  \log\left(  \bar{F}(K)+\int_{0}^{K}e^{\theta^{n}%
(t,y)}dF(y)\right)  \right)
\]
converges to
\[
\psi_{N}\left(  \log\left(  \int_{0}^{\infty}e^{\theta^{n}(t,y)}dF(y)\right)
\right)
\]
uniformly on $t\in\lbrack0,T]$. In sum, \underline{$\lim$}$_{K\rightarrow
\infty}I_{K}^{n}(\bar{q})=I^{n}(\bar{q})$ as $K\rightarrow\infty$. Therefore,
there exists $K_{n}$ such that $I_{K_{n}}^{n}(\bar{q})\geq I^{n}(\bar{q}%
)-1/n$. Recall that $I_{K_{n}}^{n}(\bar{q})\leq I_{K_{n}}(\bar{q})$ and
consequently we obtain
\[
I^{n}(\bar{q})-\frac{1}{n}\leq I_{K_{n}}(\bar{q})\leq I(\bar{q}).
\]
Since $I_{K}(\bar{q})$ increases in $K$, we have $I_{K}(q)\nearrow I(q)$ as
claimed.\newline For part ii), when $I(\bar{q})=\infty$, there are two cases:
a) $\bar{q}$ is not absolutely continuous, and b) $\bar{q}$ is absolutely
continuous. Case b) in turn is divided into two subcases: b.1) $\partial
^{2}\bar{q}(t,y)/\partial t\ \partial y$ is not integrable over $\mathcal{D}$,
and b.2) $\partial^{2}\bar{q}(t,y)/\partial t\ \partial y$ is integrable over
$\mathcal{D}$. We shall proceed to analyze all these cases now. For Case a).
We can construct a projection $p_{\kappa}$ with $I_{K}(p_{\kappa}(\bar{q}))>M$
as we did in the proof of Lemma \ref{projection limit}. For Case b), we have
that $\bar{q}$ is absolutely continuous, but%
\[
\sup_{\theta(\cdot,\cdot)\in C_{b}[0,T]\times\lbrack0,\infty)}\int_{0}%
^{T}\left[  \int_{0}^{\infty}\theta(t,y)\left(  -\frac{\partial^{2}}{\partial
t\ \partial y}\bar{q}(t,y+t)\right)  dy-\psi_{N}\left(  \log\left(  \int
_{0}^{\infty}e^{\theta(t,y)}dF(y)\right)  \right)  \right]  dt=\infty,
\]
so we proceed to study case b.1): Assume that $\partial^{2}\bar{q}%
(t,y)/\partial t\ \partial y$ is not integrable on $\mathcal{D}$. We shall
assume that
\begin{equation}
\int_{0}^{T}\int_{t}^{\infty}\left(  -\frac{\partial^{2}}{\partial t\ \partial
y}\bar{q}(t,y)\right)  ^{+}dydt=\infty\label{Expression2}%
\end{equation}
(if this integral is finite, then integral of the negative part must diverge
and the analysis that follows next is identical). As in the proof of Lemma
\ref{projection limit}, given a projection $\kappa$ induced by $0\leq
t_{1}<t_{2}<...<t_{m}\leq T$ and $0\leq y_{0}<y_{1}<...<y_{n+1}$, define%
\begin{align}
\delta_{ij}\left(  \kappa\right)   &  :=\bar{q}(t_{i},y_{j-1})-\bar{q}%
(t_{i},y_{j})-\bar{q}(t_{i-1},y_{j-1})+\bar{q}(t_{i-1},y_{j})\nonumber\\
&  =-\int_{t_{i-1}}^{t_{i}}\int_{y_{j-1}}^{y_{j}}\frac{\partial^{2}}{\partial
y\partial t}\bar{q}(t,y)dydt, \label{Expression1}%
\end{align}
as long as $y_{j-1}\geq t_{i-1}$ (otherwise, if $y_{j-1}<t_{i-1}$,
$\delta_{ij}\left(  \kappa\right)  =0$). Then, from (\ref{Expression2}) and
(\ref{Expression1}), it follows easily that for any $M$, there exists a
partition $\kappa$ such that
\[
\sum_{i,j}\delta_{ij}\left(  \kappa\right)  >M+T\psi_{N}(1).
\]
Therefore, for large enough $K$,
\[
I_{K}(p_{\kappa}\left(  \bar{q}\right)  )\geq\sum_{i,j}\delta_{ij}\left(
\kappa\right)  -T\psi_{N}(1)>M.
\]
Now, for case b.2) suppose that $\partial^{2}\bar{q}(t,y)/\partial t\ \partial
y$ is integrable on $\mathcal{D}$. We can find $\theta(t,y)$ such that
\[
3M<\int_{0}^{T}\int_{t}^{\infty}\theta(t,y-t)\left(  -\frac{\partial^{2}%
}{\partial t\ \partial y}\bar{q}(t,y)\right)  dy-\psi_{N}\left(  \log\left(
\int_{0}^{\infty}e^{\theta(t,y)}dF(y)\right)  \right)  dt<\infty.
\]
Following the same line of reasoning as in the proof of part i) we can
conclude that there exists $K>0$ such that $I_{K}(\bar{q})>2M$. According to
Dawson-Gartner Theorem, $I_{K}(\bar{q})=\sup I_{K}(p_{\kappa}(\bar{q}))$ where
the supremum is taken over all projections restricted to $\{t\in[0,T],0\leq
y\leq t+K\}$. As a result, there exists some projection $p_{\kappa}$ such that
$I_{K}(p_{\kappa}(\bar{q}))>M$ and hence we are done. \newline Now we turn to
part iii). So far we proved that for any $\bar{q}$ and $M>0$, we can find a
projection $p_{\kappa}$ such that $I_{K}(p_{\kappa}(\bar{q}))>2M$. As
discussed in the proof of Lemma \ref{projection limit}, we have
\[
\sup_{\{\theta_{ij}:1\leq i\leq m,1\leq j\leq n+1\}}\sum_{i=1}^{m}\sum
_{j=1}^{n+1}\theta_{ij}\delta_{ij}(\kappa)-\sum_{i=1}^{m}\int_{t_{i-1}}%
^{t_{i}}\psi_{N}^{(K)}\left(  \log\sum_{j=1}^{n+1}e^{\theta_{ij}}%
P(y_{j-1}-u<V_{1}^{(K)}\leq y_{j}-u)\right)  du>2M.
\]
where $\delta_{ij}(\kappa)$ is induced by the projection $p_{\kappa}$. By
definition, there exists some $\theta_{ij}$ such that
\[
\sum_{i=1}^{m}\sum_{j=1}^{n+1}\theta_{ij}\delta_{ij}(q)-\sum_{i=1}^{m}%
\int_{t_{i-1}}^{t_{i}}\psi_{N}^{(K)}\left(  \log\sum_{j=1}^{n+1}e^{\theta
_{ij}}P(y_{j-1}-u<V_{1}^{(K)}\leq y_{j}-u)\right)  du>3M/2.
\]
For all $\varepsilon>0$ and $\hat{q}\in B_{\varepsilon}(p)$, we have
\[
|\delta_{ij}(q)-\delta_{ij}(\hat{q})|\leq4\varepsilon.
\]
Hence for $\varepsilon=M/(8\sum_{i,j}|\theta_{ij}|)$ and all $\hat{q}\in
B_{\varepsilon}(q)$, we have
\begin{align*}
I(p_{\kappa}(\hat{q}))  &  \geq\sum_{i=1}^{m}\sum_{j=1}^{n+1}\theta_{ij}%
\delta_{ij}(\hat{q})-\sum_{i=1}^{m}\int_{t_{i-1}}^{t_{i}}\psi_{N}^{(K)}\left(
\log\sum_{j=1}^{n+1}e^{\theta_{ij}}P(y_{j-1}-u<V_{1}^{(K)}\leq y_{j}%
-u)\right)  du\\
&  \geq\sum_{i=1}^{m}\sum_{j=1}^{n+1}\theta_{ij}\delta_{ij}(q)-\sum_{i=1}%
^{m}\int_{t_{i-1}}^{t_{i}}\psi_{N}^{(K)}\left(  \log\sum_{j=1}^{n+1}%
e^{\theta_{ij}}P(y_{j-1}-u<V_{1}^{(K)}\leq y_{j}-u)\right)  du-4\varepsilon
\sum_{i=1}^{m}\sum_{j=1}^{n+1}|\theta_{ij}|\\
&  >M.
\end{align*}
Thus we conclude the result.
\end{proof}

\begin{proof}
[Proof of Lemma \ref{arrivalsK}]Let $N_{\lambda}^{\left(  K\right)  }(T)$ be
the total number of arrivals from time 0 up to $T$ with service time longer
than $K$, under the $\lambda$-scaled system. Then following \cite{Glynn95a}
(or as in the proof of Lemma \ref{tightness}) we have that%
\[
\lim_{\lambda\rightarrow\infty}\frac{1}{\lambda}\log Ee^{\theta N_{\lambda
}^{\left(  K\right)  }(T)}=T\psi_{N}(\log(e^{\theta}\bar{F}(K)+F(K))).
\]
Chernoff's bound yields%
\[
P(N_{\lambda}^{\left(  K\right)  }(T)>\lambda\varepsilon)\leq\exp
\{-\theta\lambda\varepsilon+\lambda T\psi_{N}(\log(e^{\theta}\bar
{F}(K)+F(K)))+o(\lambda)\}.
\]
Hence
\begin{align*}
&  \overline{\lim}_{\lambda\rightarrow\infty}\frac{1}{\lambda}\log P_{\lambda
}(N_{\lambda}^{\left(  K\right)  }(T)>\lambda\varepsilon)\\
\leq &  -\theta\varepsilon+T\psi_{N}(\log(e^{\theta}\bar{F}(K)+F(K)))
\end{align*}
Letting $K\rightarrow\infty$ gives
\[
\lim_{K\rightarrow\infty}\overline{\lim}_{\lambda\rightarrow\infty}\frac
{1}{\lambda}\log P_{\lambda}(N_{\lambda}^{\left(  K\right)  }(T)>\lambda
\varepsilon)\leq-\theta\varepsilon
\]
Since $\theta$ can be arbitrarily large, the result follows.

\end{proof}

\bigskip

\section{Examples\label{SectionExamples}}

This section is devoted to two examples that apply the large deviations
principle that we have developed in the previous sections. The first example
is on the most likely path to overflow in a loss queue, while the second
example is on the ruin of a large life insurance portfolio that embeds an
infinite server queue with service cost. \bigskip

\noindent\textbf{Example 1.} (\emph{Finite-horizon maximum of queue length
process for $M/G/\infty$}) Consider an $M/G/\infty$ queue with Poisson
arrivals with rate $\lambda$. Suppose that the service times have a density
$f\left(  \cdot\right)  $ with respect to the Lebesgue measure. The system
initially starts empty.

We want to find the optimal large deviations sample path to attain the event
$\{\max_{0\leq t\leq T}\bar{Q}_{\lambda}(t,t)/\lambda\geq x\}$, for fixed $T$
and $x$, as $\lambda\rightarrow\infty$; this event corresponds precisely to
the event of observing a loss in a queue with $\lambda x$ servers, no waiting
room, starting empty. Note that $g\left(  \bar{q}\right)  :=\max_{0\leq t\leq
T}\bar{q}(t,t)$ is a continuous function under the uniform norm, so the
contraction principle is directly applicable.

We impose the condition that $\int_{0}^{T}\bar{F}(t)dt<x$. This condition
implies that the probability for the queue to reach $\lambda x$ decreases
exponentially fast as $\lambda\rightarrow\infty$ (Such condition will be clear
when we solve the constrained optimization in a moment).

To proceed, let us first observe that $\psi_{N}(\theta)=e^{\theta}-1$. The
maximization problem in \eqref{R} can be solved and the rate function is
immediately recognized as
\[
\int_{0}^{T}\int_{t}^{\infty}\left(  -\frac{\partial^{2}}{\partial t\partial
y}\bar{q}(t,y)\right)  \left(  \log\left(  -\frac{\partial^{2}}{\partial
t\partial y}\bar{q}(t,y)/f(y-t)\right)  -1\right)  dydt+T,
\]
which is easily seen to be a convex function of $\partial^{2}\bar
{q}(t,y)/\partial t\partial y$. To find the optimal sample path amounts to
solving the minimization problem
\begin{equation}%
\begin{array}
[c]{ll}%
\text{min} & \int_{0}^{T}\int_{t}^{\infty}\left(  -\frac{\partial^{2}%
}{\partial t\partial y}\bar{q}(t,y)\right)  \left(  \log\left(  -\frac
{\partial^{2}}{\partial t\partial y}\bar{q}(t,y)/f(y-t)\right)  -1\right)
dydt+T\\
\text{subject to} & \max_{0\leq s\leq T}\int_{0}^{s}\int_{s}^{\infty}\left(
-\frac{\partial^{2}}{\partial t\partial y}\bar{q}(t,y)\right)  dydt\geq x,
\end{array}
\label{min problem}%
\end{equation}
which is a convex optimization problem. The quantity $\int_{0}^{s}\int
_{s}^{\infty}\left(  -\partial^{2}\bar{q}(t,y)/\partial t\partial y\right)
dydt$ is equal to $\bar{q}(s,s)$ when $\bar{q}$ is absolutely continuous, and
$\bar{q}(s,s)$ represents the scaled queue length process at time $s$.

To solve \eqref{min problem}, we first consider a fixed $s$ in the constraint
and then optimize over $s$. When considering $s$ fixed we replace the
constraint in \eqref{min problem} by $\bar{q}(s,s)\geq x$. Under this new
constraint, it suffices to look at the time 0 to $s$ in the objective
function, that is, we now solve%
\begin{equation}%
\begin{array}
[c]{ll}%
\text{min} & \int_{0}^{s}\int_{t}^{\infty}\left(  -\frac{\partial^{2}%
}{\partial t\partial y}\bar{q}(t,y)\right)  \left(  \log\left(  \frac
{-\frac{\partial^{2}}{\partial t\partial y}\bar{q}(t,y)}{f(y-t)}\right)
-1\right)  dydt+s\\
\text{subject to} & \int_{0}^{u}\int_{u}^{\infty}\left(  -\frac{\partial^{2}%
}{\partial t\partial y}\bar{q}(t,y)\right)  dydt\geq x.
\end{array}
\label{easy min problem}%
\end{equation}
The solution to \eqref{min problem} is then the optimal sample path from
\eqref{easy min problem}, among $0\leq u\leq T$, that gives the smallest objective.

We now consider \eqref{easy min problem}. Introducing a Lagrange multiplier
$\mu\geq0$, we minimize
\[
\int_{0}^{u}\left(  \int_{u}^{\infty}\left(  -\frac{\partial^{2}}{\partial
t\partial y}\bar{q}(t,y)\right)  \left(  \log\left(  \frac{-\frac{\partial
^{2}}{\partial t\partial y}\bar{q}(t,y)}{f(y-t)}\right)  -1\right)  dy-\mu
\int_{s}^{\infty}\left(  -\frac{\partial^{2}}{\partial t\partial y}\bar
{q}(t,y)\right)  dy\right)  dt.
\]
By a formal application of Euler-Lagrange equations, we differentiate the
integrand with respect to $-\partial^{2}\bar{q}(t,y)/\partial t\partial y$ to
get
\[
\left\{
\begin{array}
[c]{ll}%
\log\left(  \frac{-\partial^{2}\bar{q}(t,y)/\partial t\partial y}%
{f(t-y)}\right)  =0 & \text{\ for\ }t\leq y\leq u\\
\log\left(  \frac{-\partial^{2}\bar{q}(t,y)/\partial t\partial y}%
{f(t-y)}\right)  -\mu=0\text{\ } & \text{\ for\ }y>u
\end{array}
\right.
\]
which gives
\[
-\frac{\partial^{2}}{\partial t\partial y}\bar{q}(t,y)=\left\{
\begin{array}
[c]{ll}%
f(y-t) & \text{\ for\ }t\leq y\leq u\\
\mu f(y-t) & \text{\ for\ }y>u
\end{array}
\right.
\]
for some $\mu\geq1$ (we replace $e^{\mu}$ by another dummy $\mu$ for
convenience). Complementary slackness then implies
\[
\int_{0}^{u}\int_{u}^{\infty}(-\partial^{2}\bar{q}(t,y)/\partial t\partial
y)dydt=\int_{0}^{u}\int_{u}^{\infty}\mu f(y-t)dydt=x,
\]
which in turn gives
\[
\mu=\frac{x}{\int_{0}^{u}\bar{F}(t)dt}%
\]
(note that we have assumed $\int_{0}^{u}\bar{F}(t)dt<x$ and so the condition
$\mu>1$ is satisfied). As a result
\[
-\frac{\partial^{2}}{\partial t\partial y}\bar{q}(t,y)=\left\{
\begin{array}
[c]{ll}%
f(y-t) & \text{\ for\ }t\leq y\leq u\\
\frac{xf(y-t)}{\int_{0}^{u}\bar{F}(t)dt} & \text{\ for\ }y>u.
\end{array}
\right.
\]

The optimal sample path $\bar{q}(t,y)$ leading to the constraint $\bar
{q}(u,u)\geq x$ is given by
\begin{align*}
&  \bar{q}(t,y)= \int_{0}^{t}\int_{y}^{\infty}\left(  -\frac{\partial^{2}%
}{\partial t\partial y}\bar{q}(s,w)\right)  dwds\\
&  =\int_{0}^{t}\left(  \int_{y\wedge u}^{u}f(w-s)dw+\int_{u\vee y}^{\infty
}\frac{xf(w-s)}{\int_{0}^{u}\bar{F}(r)dr}dw\right)  ds
\end{align*}
Transforming into $q(t,y)=\bar{q}(t,y+t)$ and some simple calculus gives the
optimal sample path
\[
q(t,y)=\int_{y}^{y+t}\bar{F}(s)ds-\int_{u-t}^{u}\bar{F}(s)ds+\frac{x\int
_{u-t}^{u}\bar{F}(s)ds}{\int_{0}^{u}\bar{F}(r)dr},\text{ for }y+t\leq u,
\]
and
\[
q(t,y)=\frac{x\int_{y}^{y+t}\bar{F}(s)ds}{\int_{0}^{u}\bar{F}(r)dr},\text{ for
}y+t>u.
\]

In connection to our discussion about the direct rate function representation
in terms of $Q_{\lambda}/\lambda$ (see equation (\ref{Qrate}), in Section
\ref{SectionGuessing}), one can check that $\partial{q}(t,y)/\partial y$ is
not continuous on the line $y=t$ and therefore $\partial^{2}q(t,y)/\partial
y^{2}$ does not exists though $I(\bar{q})$ is finite.

Note also that the objective is
\begin{align}
&  \int_{0}^{u}\left(  -\int_{t}^{u}f(y-t)dy+\int_{u}^{\infty}\frac
{xf(y-t)}{\int_{0}^{u}\bar{F}(t)dt}\left(  \log\left(  \frac{x}{\int_{0}%
^{u}\bar{F}(t)dt}\right)  -1\right)  dy\right)  dt+T\nonumber\\
&  =\int_{0}^{u}\bar{F}(t)dt+\left(  \log\left(  \frac{x}{\int_{0}^{s}\bar
{F}(t)dt}\right)  -1\right)  x. \label{rate Poisson}%
\end{align}

This is the rate function corresponding to the probability $P(\bar{Q}%
_{\lambda}(u,u)\geq\lambda x)=P(Q_{\lambda}(u,0)\geq\lambda x)$, where
$Q_{\lambda}(u,0)$ is the queue length at time $u$. This rate of decay is
consistent with direct calculation using the fact that $Q_{\lambda}(u,0)$ is a
Poisson random variable with rate $\lambda\int_{0}^{u}\bar{F}(t)dt$, which
gives
\[
P(Q(s)>\lambda x)=\sum_{n\geq\lambda x}e^{-\lambda\int_{0}^{u}\bar{F}%
(t)dt}\left(  \lambda\int_{0}^{u}\bar{F}(t)dt\right)  ^{n}/n!.
\]
For a consistency check, our result here can in fact recover the large
deviations for the arrival process itself. If one changes the constraint in
\eqref{easy min problem} to $\bar{q}(u,0)=\int_{0}^{u}\int_{0}^{\infty}\left(
-\frac{\partial^{2}}{\partial t\partial y}\bar{q}(t,y)\right)  dydt\geq x$,
the optimal value of \eqref{easy min problem} then becomes $x[\log(x/s)-1]+s$,
which coincides with the exponential decay rate of $P($Poisson$(\lambda
s)>\lambda x)$ as $\lambda\rightarrow\infty$.

Figures \ref{FigureOver1} and \ref{FigureOver2} illustrate both the law of
large numbers (i.e. the typical path) and the most likely path to the overflow
event $\max_{0\leq u\leq T}\bar{Q}_{\lambda}(u,u)/\lambda\geq x$ for $T=1$,
$x=2$. The underlying service time distribution is uniform in the interval
$[0,1]$. We can see that the optimal path of $Q(t,y)$ increases gradually over
time to overflow at time 1.

\begin{figure}[th]
\centering
\subfloat[][$Q(t,y)$ for the most likely path to overflow (surface)]{
		\includegraphics[scale=0.5]{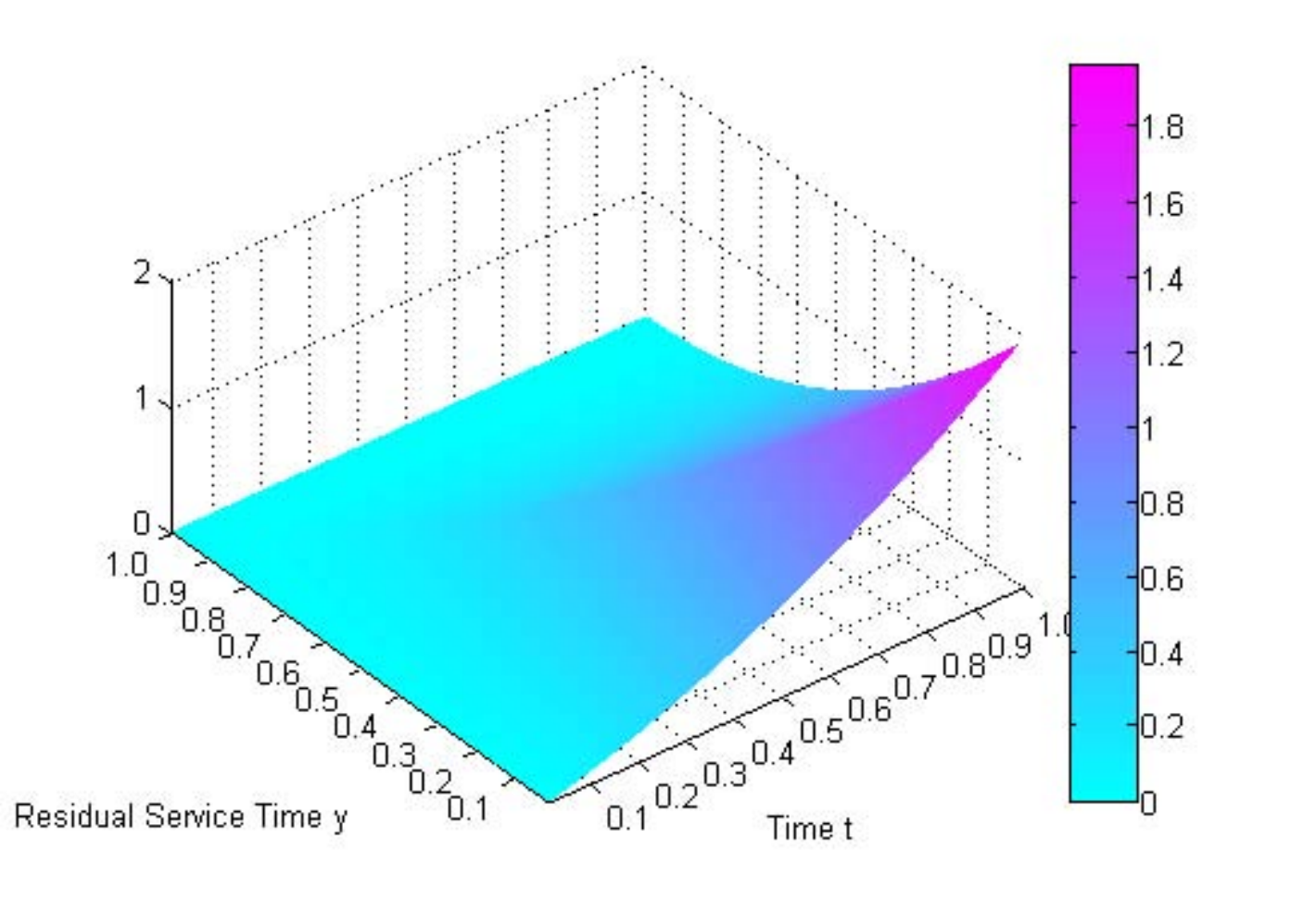}
		\label{Over1a}
} \subfloat[][$Q(t,y)$ for the unconditional path (surface)]{
\includegraphics[scale=0.25]{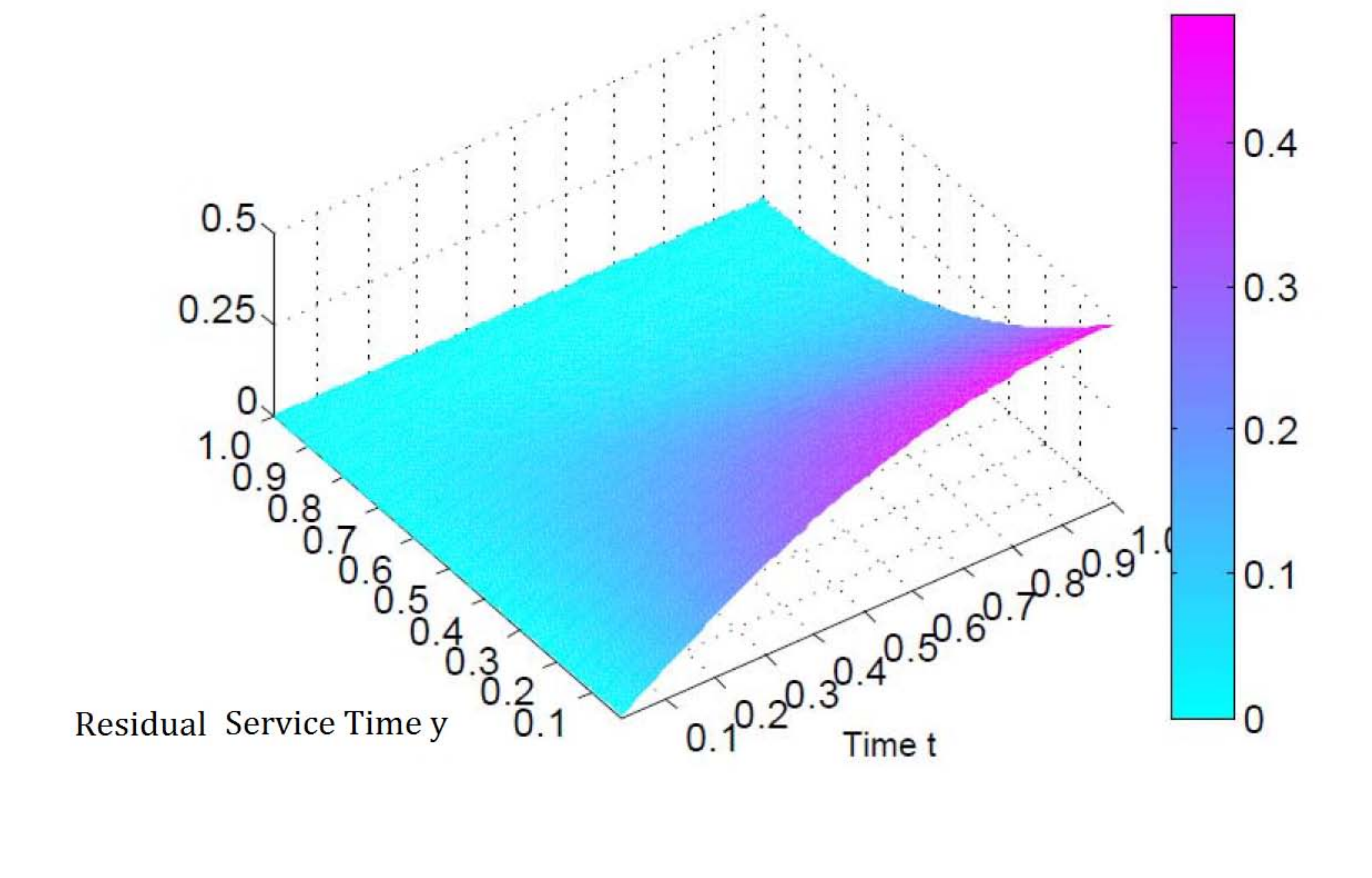}
		\label{Over1b}
}\caption{Surface plots of the asymptotic surface $Q_{\lambda}(t,y)/\lambda$,
as $\lambda$ increases, both an optimal (most likely) path leading to
overflow, and the unconditional path. }%
\label{FigureOver1}%
\end{figure}

\begin{figure}[th]
\centering
\subfloat[][$Q(t,y)$ for the most likely path to overflow (contour)]{
		\includegraphics[scale=0.4]{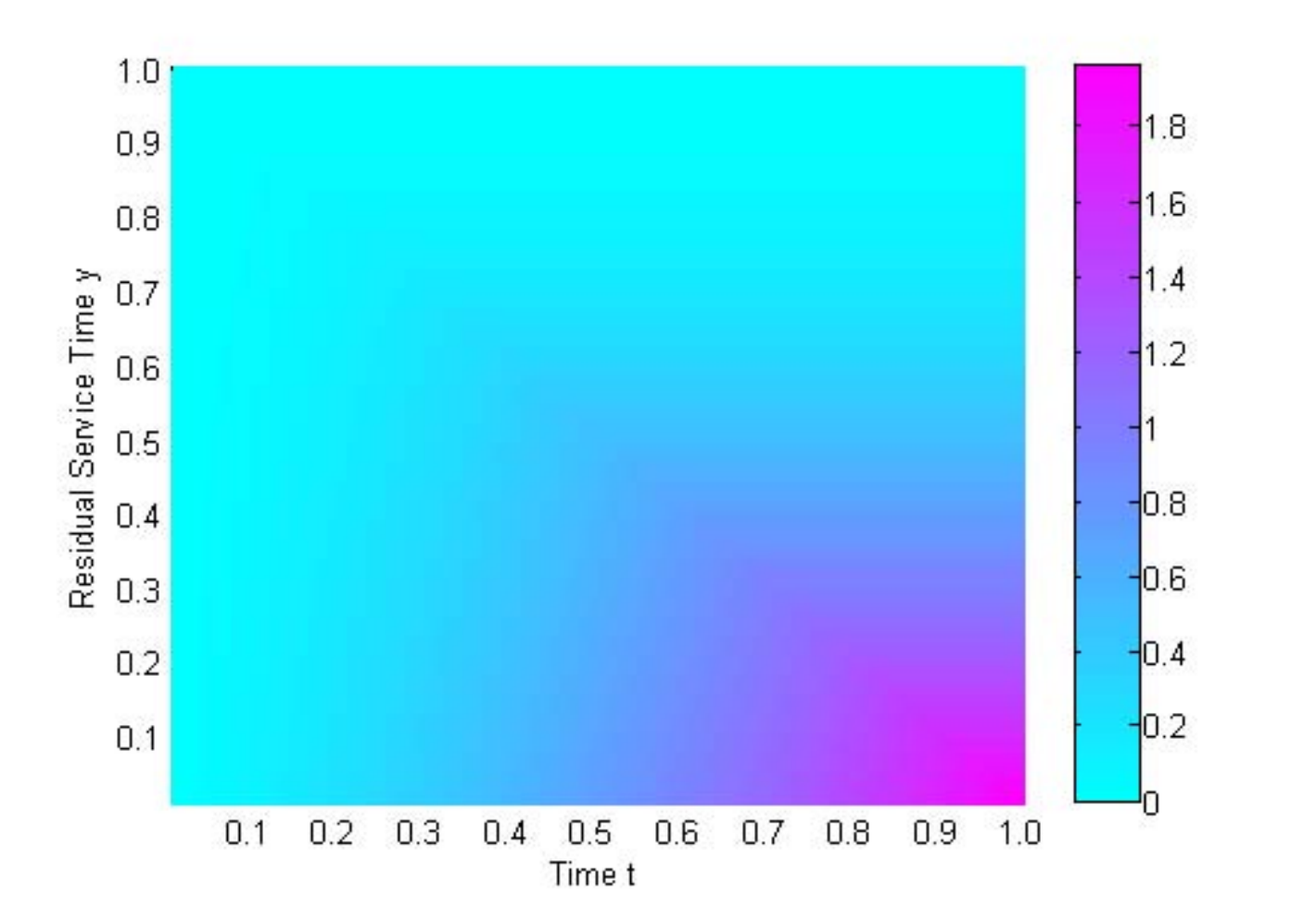}
		\label{over2a}
} \subfloat[][$Q(t,y)$ for the unconditional path (contour)]{
\includegraphics[scale=0.3]{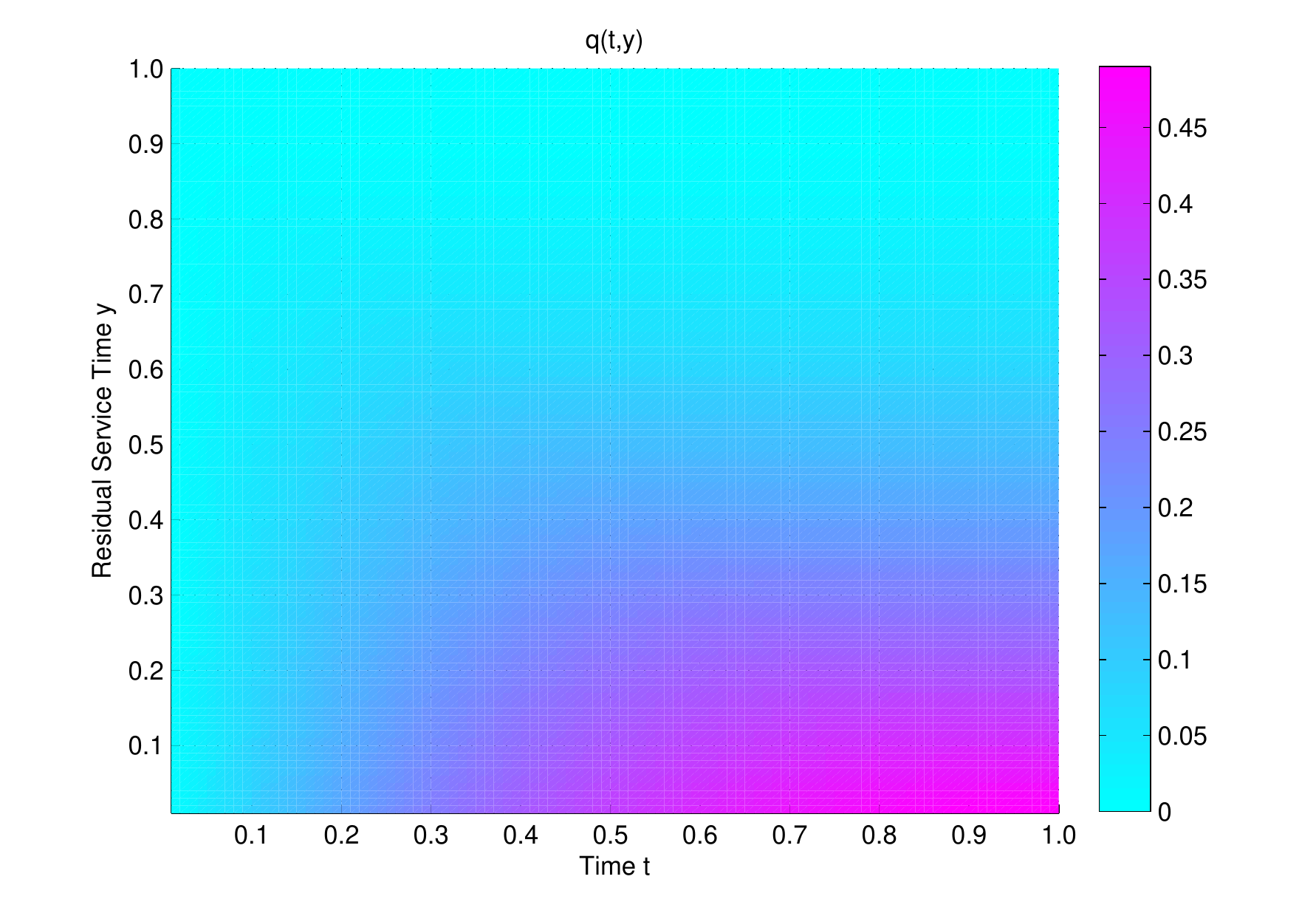}
		\label{over2b}
}\caption{Contour plots of the asymptotic surface $Q_{\lambda}(t,y)/\lambda$,
as $\lambda$ increases, both an optimal (most likely) path leading to
overflow, and the unconditional path.}%
\label{FigureOver2}%
\end{figure}

It is easy to see that since we assume $\int_{0}^{T}\bar{F}(u)du<x$, the rate
function \eqref{rate Poisson} is non-decreasing in $s$, and as a result an
optimal time horizon is $T$. If the service time has bounded support $[0,M]$
with $M<T$, then the selection of any time $s\in\lbrack M,T]$ will give an
optimal sample path.
\bigskip

\textbf{Example 2.} (\emph{Insurance risk process}) The net reserve of a life
insurance company consists of the premium collected from policyholders,
deducted by the benefit paid to policyholders in the event of deaths; often
all these payments are discounted at zero in order to recognize the value of
money in time. When policyholders arrive at the insurance company over time
(an arrival is interpreted as the moment when a contract is signed), one can
model the net assets of the insurer as a function of the underlying arrivals
and death events of policyholders. Specifically, we shall assume that
policyholders arrive according to a Poisson process with rate $\lambda$, and
that the time-until-death upon arrival of the policyholders are independent
and identically distributed. Moreover, we assume that the time-until-death
upon arrival has density $f(\cdot)$, distribution function $F(\cdot)$, and
tail distribution $\bar{F}(\cdot)$. The time-until-death in this setting can
be thought as the service time in the queueing context. We shall assume
without loss of generality that the initial net reserve of the company is zero.

It is often more convenient to work with the negative net reserve process,
also known as the \textit{aggregate loss process}, defined as the total
benefit that the insurer has paid up to time $t$, minus the total premium
received up to time $t$. For a policyholder who arrives at time $A_{i}$, and
who dies at time $A_{i}+V_{i}<t$, the payoff by the insurer, discounted at
time zero, is denoted $h_{1}(A_{i},A_{i}+V_{i})$; here $A_{i}$ and $V_{i}$ are
the arrival time and time-until-death at the time of arrival of the
policyholder. This quantity, $h_{1}(s,y)$, for $y\geq s$, captures the benefit
paid at $y$ minus the accumulated premium collected from time $s$ to $y$. On
the other hand, for a policyholder who has arrived prior to $t$, at time
$A_{i}$, and who is still alive at time $t$, the payoff from the insurer to
the policyholder is $h_{2}(A_{i},t)$ (typically $h_{2}\left(  A_{i},t\right)
$ will be negative as it represents premium that are paid to the insurer, so
the payoff is negative). Here $h_{2}(s,t)$, for $t\geq s$, captures the
premium accumulated from $s$ up to the present time $t$, discounted to obtain
the net present value at time zero.

Consider, for instance, the setting of whole life insurance policies. That is,
policies that pay a benefit $b$ to the family of the policyholder, at the time
of eventual death, in exchange of a premium which is paid at rate $p$
continuously in time during all the time the policy was held, from arrival, up
until the time of death. If the interest rate (or force of interest as it is
known in the insurance setting) is constant equal to $\delta>0$, then
\[
h_{1}(s,y)=be^{-\delta y}-\int_{s}^{y}pe^{-\delta r}dr=be^{-\delta y
}-p(e^{-\delta s}-e^{-\delta y})/\delta,
\]
and
\[
h_{2}(s,t)=-\int_{s}^{t}pe^{-\delta r}dr=-p(e^{-\delta s}-e^{-\delta
t})/\delta.
\]


The aggregate loss process, $S_{\lambda}\left(  t\right)  $, is represented as
the net present value of the sum of the payoffs for all policyholders who
arrive before $t$ and it is given by%

\begin{align*}
S_{\lambda}\left(  t\right)   &  =\sum_{i=1}^{N_{\lambda}(t)}\left(  I\left(
A_{i}+V_{i}\leq t\right)  h_{1}(A_{i},A_{i}+V_{i})+I(A_{i}+V_{i}>t)h_{2}%
(A_{i},t)\right) \\
&  =\int_{0}^{t}\int_{s}^{t}h_{1}(s,y)d\bar{Q}_{\lambda}(ds,dy)+\int_{0}%
^{t}\int_{t}^{\infty}h_{2}(s,y)d\bar{Q}_{\lambda}(ds,dy).
\end{align*}
We claim that $S_{\lambda}\left(  \cdot\right)  $ is a continuous function of
$\bar{Q}_{\lambda}\left(  \cdot\right)  $ under the uniform topology on
$\mathcal{D}$. In order to see this, define $D_{\lambda}\left(  t\right)  $ to
be the number of departures by time $t$, that is,
\begin{equation}
D_{\lambda}\left(  t\right)  =N_{\lambda}\left(  t\right)  -\bar{Q}_{\lambda
}\left(  t,t\right)  =\bar{Q}_{\lambda}\left(  t,0\right)  -\bar{Q}_{\lambda
}\left(  t,t\right)  . \label{Rep0}%
\end{equation}
Note that $D_{\lambda}\left(  \cdot\right)  $ and $N_{\lambda}\left(
\cdot\right)  $ are clearly continuous functions of $\bar{Q}_{\lambda}\left(
\cdot\right)  $. Moreover, we have that
\[
\sum_{i=1}^{N_{\lambda}(t)}I\left(  A_{i}+V_{i}\leq t\right)  h_{1}%
(A_{i},A_{i}+V_{i})=\int_{0}^{t}\int_{0}^{t}h_{1}\left(  s,u\right)
D_{\lambda}\left(  du\right)  N_{\lambda}\left(  ds\right)  ,
\]
and therefore%
\[
S_{\lambda}\left(  t\right)  =\int_{0}^{t}\int_{0}^{t}h_{1}\left(  s,u\right)
D_{\lambda}\left(  du\right)  N_{\lambda}\left(  ds\right)  +\int_{0}^{t}%
\int_{t}^{\infty}h_{2}(s,y)\bar{Q}_{\lambda}(ds,dy).
\]
Now, integration by parts shows that%
\begin{align}
&  \int_{0}^{t}\int_{0}^{t}h_{1}\left(  s,u\right)  D_{\lambda}\left(
du\right)  N_{\lambda}\left(  ds\right) \nonumber\\
&  =\int_{0}^{t}\int_{0}^{t}\frac{\partial^{2}}{\partial u\partial s}%
h_{1}\left(  s,u\right)  D_{\lambda}\left(  u\right)  N_{\lambda}\left(
s\right)  dsdu-N_{\lambda}\left(  t\right)  \int_{0}^{t}D_{\lambda}\left(
u\right)  \frac{\partial}{\partial u}h_{1}\left(  t,u\right)  du\nonumber\\
&  -D_{\lambda}\left(  t\right)  \int_{0}^{t}\frac{\partial}{\partial s}%
h_{1}\left(  s,t\right)  N_{\lambda}\left(  s\right)  ds+D_{\lambda}\left(
t\right)  N_{\lambda}\left(  t\right)  h_{1}\left(  t,t\right)  . \label{Rep1}%
\end{align}
A similar development yields%
\begin{align}
\int_{0}^{t}\int_{t}^{\infty}h_{2}(s,y)\bar{Q}_{\lambda}(ds,dy)  &  =\int
_{t}^{\infty}\int_{0}^{t}\bar{Q}_{\lambda}(s,y)\frac{\partial^{2}h_{2}\left(
s,y\right)  }{\partial s\partial y}dsdy-\int_{t}^{\infty}\bar{Q}_{\lambda
}(t,y)\frac{\partial h_{2}\left(  t,y\right)  }{\partial y}dy\nonumber\\
&  +\int_{0}^{t}\frac{\partial h_{2}(s,t)}{\partial s}\bar{Q}_{\lambda
}(s,t)ds-\bar{Q}_{\lambda}(t,t)h_{2}\left(  t,t\right)  . \label{Rep2}%
\end{align}
It is now not difficult to see from (\ref{Rep1}) and (\ref{Rep2}) that indeed
$S_{\lambda}\left(  \cdot\right)  $ is a continuous function of $Q_{\lambda
}\left(  \cdot\right)  $ in the uniform topology on $[0,T]\times
\lbrack0,\infty)$.


Consider the finite-horizon ruin probability that the negative net asset of
the insurer rises above the level $\lambda x$ by time $T$. That is, the event
$\{\max_{t\in\lbrack0,T]}S_{\lambda}\left(  t\right)  /\lambda\geq x\}$. We
wish to solve for the most likely path that leads to this event and therefore,
applying our theory, we must solve the following convex calculus of variations
problem.%
\[%
\begin{array}
[c]{ll}%
\text{min} & \int_{0}^{T}\int_{t}^{\infty}\left(  -\frac{\partial^{2}%
}{\partial t\partial y}\bar{q}(t,y)\right)  \left(  \log\left(  \frac
{-\frac{\partial^{2}}{\partial t\partial y}\bar{q}(t,y)}{f(y-t)}\right)
-1\right)  dydt+T\\
\text{subject to} & \max_{0\leq u\leq T}\int_{0}^{u}\left(  \int_{t}^{u}%
h_{1}(t,y)\left(  -\frac{\partial^{2}}{\partial t\partial y}\bar
{q}(t,y)\right)  dy+\int_{u}^{\infty}h_{2}(t,u)\left(  -\frac{\partial^{2}%
}{\partial t\partial y}\bar{q}(t,y)\right)  dy\right)  dt\geq
x\label{easy min problem insurance}%
\end{array}
\]
Following the recipe of Example 1, we first consider
\[%
\begin{array}
[c]{ll}%
\text{min} & \int_{0}^{u}\int_{t}^{\infty}\left(  -\frac{\partial^{2}%
}{\partial t\partial y}\bar{q}(t,y)\right)  \left(  \log\left(  \frac
{-\frac{\partial^{2}}{\partial t\partial y}\bar{q}(t,y)}{f(y-t)}\right)
-1\right)  dydt+u\\
\text{subject to} & \int_{0}^{u}\left(  \int_{t}^{u}h_{1}(t,y)\left(
-\frac{\partial^{2}}{\partial t\partial y}\bar{q}(t,y)\right)  dy+\int
_{u}^{\infty}h_{2}(t,u)\left(  -\frac{\partial^{2}}{\partial t\partial y}%
\bar{q}(t,y)\right)  dy\right)  dt\geq
x\label{easy min problem insurance copy(1)}%
\end{array}
\]
Introducing the Lagrange multiplier $\mu\geq0$, we get
\[
-\frac{\partial^{2}}{\partial t\partial y}\bar{q}(t,y)=\left\{
\begin{array}
[c]{ll}%
f(y-t)e^{\mu h_{1}(t,y)} & \text{\ for\ }t\leq y\leq u\\
f(y-t)e^{\mu h_{2}(t,u)} & \text{\ for\ }y>u
\end{array}
\right.
\]
When $x$ is large, complementary slackness forces $\mu$ to satisfy
\begin{equation}
\int_{0}^{u}\left(  \int_{t}^{u}f(y-t)e^{\mu h_{1}(t,y)}h_{1}(t,y)dy+\bar
{F}(u-t)e^{\mu h_{2}(t,u)}h_{2}(t,u)\right)  dt=x
\label{complementary slackness}%
\end{equation}
for some $\mu>0$. Denote the integration on the left hand side by $G(\mu)$,
then
\[
G^{\prime}(\mu)=\int_{0}^{u}\left(  \int_{t}^{u}f(y-t)e^{\mu h_{1}(t,y)}%
h_{1}^{2}(t,y)dy+\bar{F}(u-t)e^{\mu h_{2}(t,u)}h_{2}^{2}(t,u)\right)  dt>0.
\]
Therefore, for given $u$, $G(\mu)$ is monotone in $\mu$. Besides, $|G^{\prime
}(\mu)|\rightarrow\infty$ as $\mu\rightarrow\infty$. As a direct consequence,
for any $x$ large enough, equation (\ref{complementary slackness}) can be
easily fit to many standard numerical solvers, and it admits a unique
solution. Given $\mu$, the optimal sample path is given by
\begin{align*}
\bar{q}(t,y)  &  =\int_{0}^{t}\int_{y}^{\infty}\left(  -\frac{\partial^{2}%
}{\partial t\partial y}\bar{q}(s,w)\right)  dwds\\
&  =\int_{0}^{t}\left(  \int_{y\wedge u}^{u}f(w-s)e^{\mu(u)h_{1}(s,w)}%
dw+\int_{u\vee y}^{\infty}f(w-s)e^{\mu(u)h_{2}(s,u)}dw\right)  ds
\end{align*}
for $y\geq t$, and hence
\begin{align*}
q(t,y)  &  =\int_{0}^{t}\int_{y+t}^{\infty}\left(  -\frac{\partial^{2}%
}{\partial t\partial y}\bar{q}(s,w)\right)  dwds\\
&  =\int_{0}^{t}\left(  \int_{(y+t)\wedge u}^{u}f(w-s)e^{\mu(u)h_{1}%
(s,w)}dw+\int_{u\vee(y+t)}^{\infty}f(w-s)e^{\mu(u)h_{2}(s,u)}dw\right)  ds
\end{align*}
for $y\geq0$. Note that here we highlight the dependence of $\mu$ on $u$.
Moreover, the rate function for the fixed-time probability is
\begin{equation}
\int_{0}^{u}\left(  \int_{t}^{u}f(y-t)e^{\mu(u)h_{1}(t,y)}(\mu(u)h_{1}%
(t,y)-1)dy+\int_{u}^{\infty}f(y-t)e^{\mu(u)h_{2}(t,u)}(\mu(u)h_{2}%
(t,u)-1)dy\right)  dt \label{rate insurance}%
\end{equation}
The optimal time horizon $u$ over $0\leq u\leq T$ is chosen to minimize \eqref{rate insurance}.

Now we consider a whole life insurance contract with benefit $b=1.5$,
continuous premium $p=1$, zero interest rate and time-until-death which
follows the uniform distribution on $[0,1]$. Our goal is to compute the
optimal sample path for ruin(x=10) before time $T=1$. We solve the constraint
equation \eqref{complementary slackness} in Matlab and obtain the optimal
$u=1$ with $\mu=2.251$.

\begin{figure}[th]
\centering
\subfloat[][$Q(t,y)$ for the most likely path to overflow (surface)]{
		\includegraphics[scale=0.3]{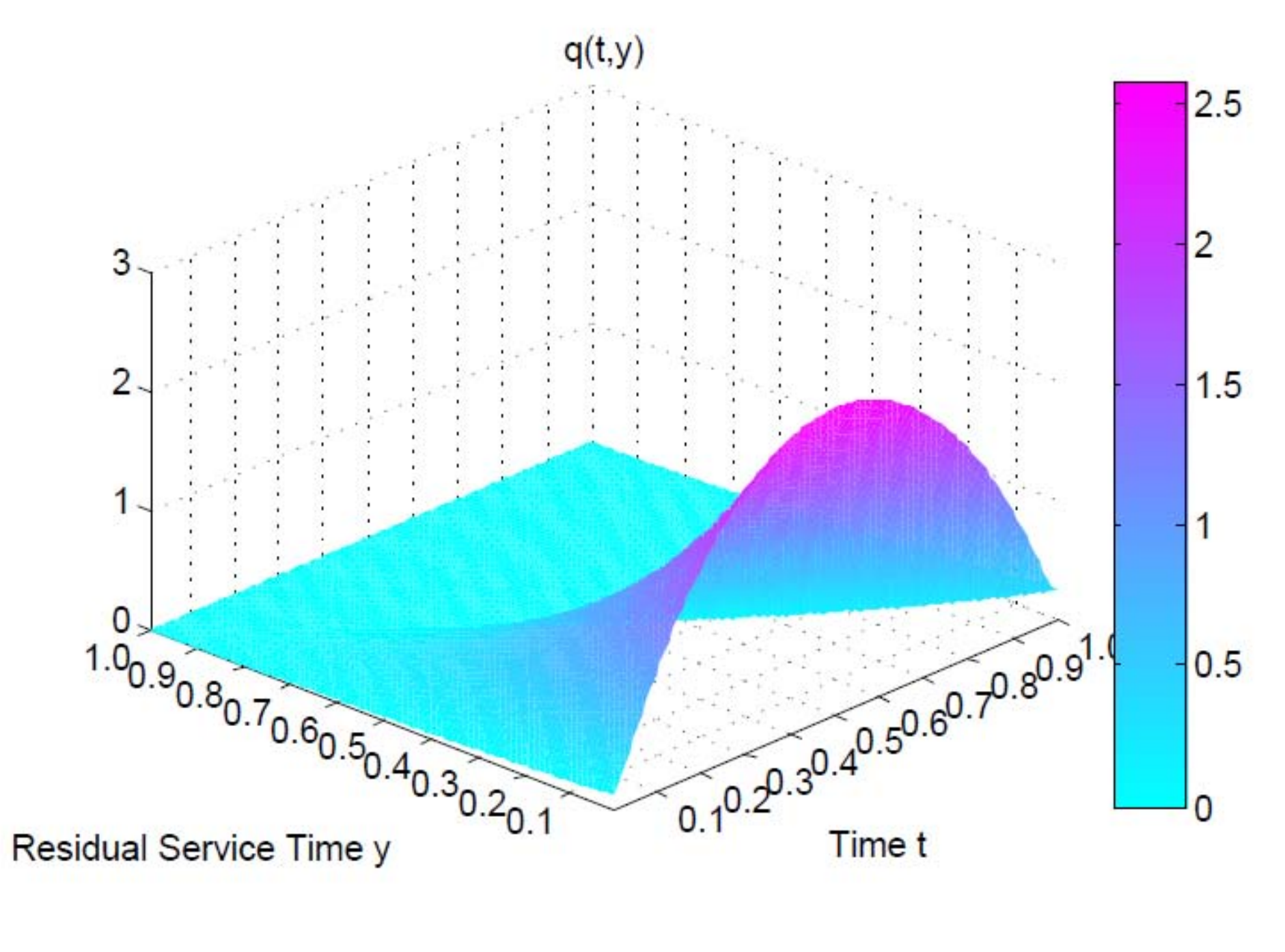}
		\label{Insurance1}
} \subfloat[][$Q(t,y)$ for the most likely path to overflow (contour)]{
\includegraphics[scale=0.3]{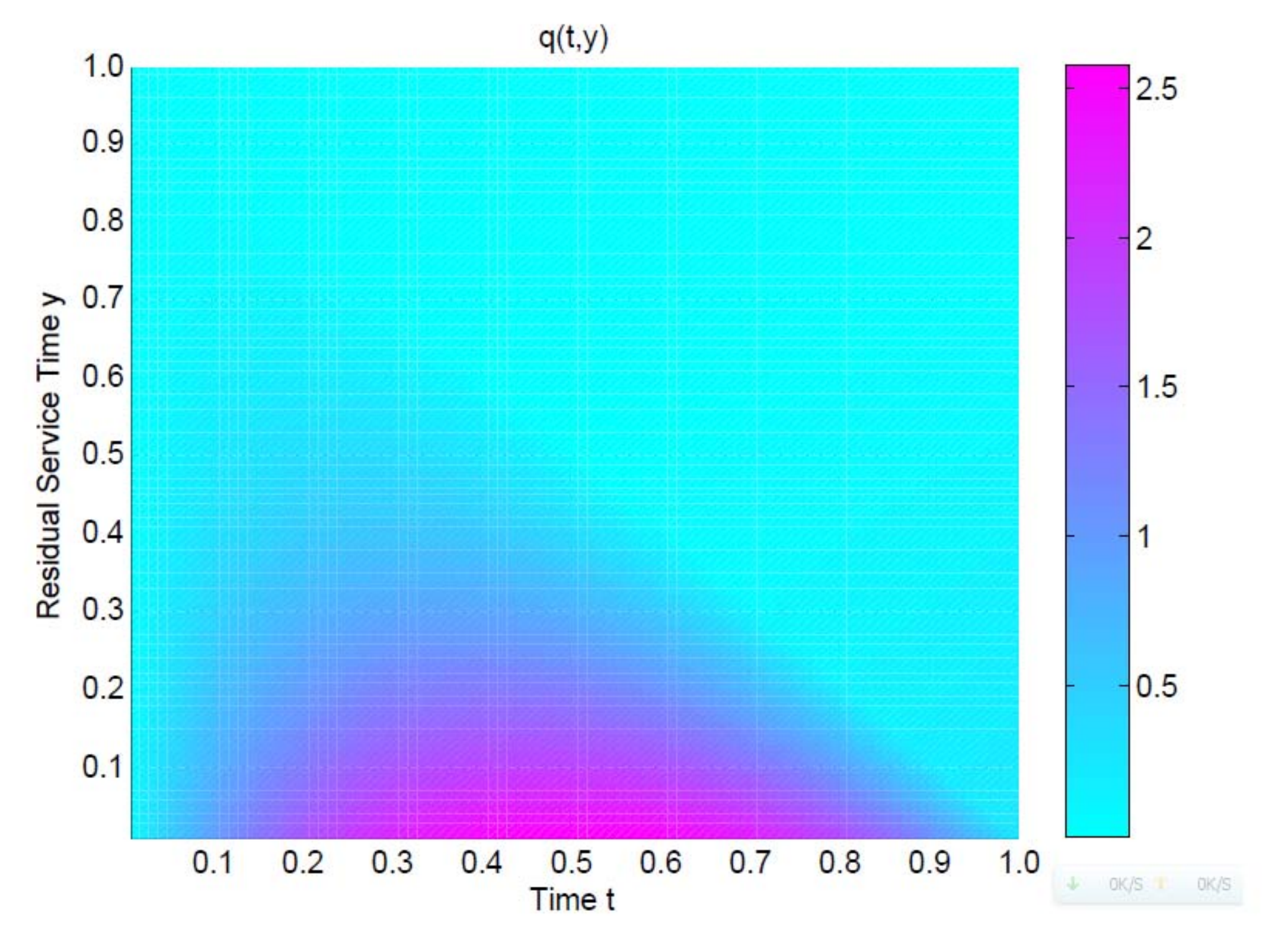}
		\label{Insurance2}
}\caption{The surface and the corresponding contour plot of the asymptotic
most likely path to ruin in a portfolio of life insurance policies.}%
\label{FigureInsurance}%
\end{figure}

In this case, we can compute the optimal path
\[
q(t,y)=\frac{1}{\mu^{2}}(e^{\mu b-\mu y}-e^{\mu b-\mu y-\mu t}-e^{\mu
b-\mu+\mu t}+e^{\mu b-\mu})+\frac{t}{\mu}e^{-\mu+\mu t}-\frac{1}{\mu^{2}%
}(e^{-\mu+\mu t}-e^{-\mu}),\text{ for }y+t\leq1,
\]
and
\[
q(t,y)=e^{-\mu(2-t-y)}\left(  \frac{1-y}{\mu}e^{\mu-\mu y}-\frac{1}{\mu^{2}%
}(e^{\mu-\mu y}-1)\right)  ,\text{ for }y+t>1.
\]
These optimal paths to ruin are shown in Figure \ref{FigureInsurance}. We just
show the conditional paths, as the unconditional path are identical to the
figures illustrated in Example 1. The optimal path here is qualitatively very
different from that of Example 1. The value of $Q(t,y)$ is the largest midway
between time 0 and 1. Intuitively, it is because it requires the smallest
``energy", or distortion from the law of large numbers, at such time point in
contributing to a large cash outflow from the insurer.

\bibliographystyle{plain}
\bibliography{prob}

\end{document}